\documentclass[12 pt]{amsart}
\usepackage{amsmath}
\usepackage{amsfonts}
\usepackage{amssymb}
\usepackage{fullpage}
\usepackage{graphicx}

\newcommand\eps{\varepsilon}
\newcommand\R{{\mathbb{R}}}

\theoremstyle{plain}
  \newtheorem{theorem}[subsection]{Theorem}
  
  \newtheorem{proposition}[subsection]{Proposition}
  \newtheorem{lemma}[subsection]{Lemma}

  \newtheorem{mainthm}{Theorem}

\theoremstyle{remark}
  \newtheorem{remark}[subsection]{Remark}

\theoremstyle{definition}
  \newtheorem{definition}[subsection]{Definition}

\begin{document}

\title[]{Conditional existence of maximizers for the Tomas-Stein inequality for the sphere}
\author{Shuanglin Shao}
\address{Department of Mathematics, University of Kansas, Lawrence, KS 66045, USA}
\email{slshao@ku.edu}
\author{Ming Wang}
\address{Department of Mathematics, University of Kansas, Lawrence, KS 66045, USA}
\email{mwang@ku.edu}

\vspace{-0.1in}
\date{\today}
\begin{abstract} The Tomas-Stein inequality for a compact subset of the sphere $\Gamma := \{x=(x', x_{d+1}): \, x\in S^d, \, |x'|\le \frac 12\}$ states that the mapping $f\mapsto \widehat{f\sigma}$ is bounded from $L^2(\Gamma,\sigma)$ to $L^{2+4/d}(\R^{d+1})$. Then conditional on a strict comparison between the best constants for the sphere and for the Strichartz inequality for the Schr\"odinger equations, we prove that there exist functions which extremize this inequality, and any extremising sequence has a subsequence which converges to an extremizer. The method is based on the refined Tomas-Stein inequality for the sphere and the profile decompositions. The key ingredient to establish orthogonality in profile decompositions is that we use Tao's sharp bilinear restriction theorem for the paraboloids beyond the Tomas-Stein range.  Similar results have been previously established by Frank, Lieb and Sabin \cite{Frank-Lieb-Sabin:2007:maxi-sphere-2d}, where they used the method of the missing mass. 
\end{abstract}

\maketitle

\section{Introduction}\label{Monkey-Eye}
The Tomas-Stein's inequality for the sphere or the adjoint Fourier restriction inequality for a compact subset of the sphere $\Gamma\subset S^d$ asserts that
\begin{equation}\label{eq-1}
\|\widehat{f\sigma}\|_{L^{2+4/d}(\R^{d+1})} \le \mathcal{R}\|f\|_{L^2(\Gamma,\sigma)}.
\end{equation} See e.g. \cite{Stein:1993}. Here the constant $\mathcal{R}>0$ is defined to be the optimal constant
\begin{equation}\label{eq-2}
\begin{split}
 \mathcal{R}:=\sup\{\|\widehat{f\sigma}\|_{L^{2+4/d}(\R^{d+1})}:\|f\|_{L^2(\Gamma,\sigma)}=
1\}, \\
\Gamma := \{x=(x', x_{d+1}): \, x\in S^d, \, |x'|\le \frac 12\}. 
\end{split}
\end{equation} and $\sigma$ denotes the surface measure on $\Gamma\subset S^d$, and the Fourier transform is defined by
\begin{equation}\label{eq-3}
\widehat{f\sigma}(\xi ):=\frac {1}{(2\pi )^{\frac {d+1}{2}}} \int_{\Gamma} e^{-i\xi \cdot x  }f(x)d\sigma (x).
\end{equation}
\begin{definition}\label{bass-hook}
A function $f\in L^2(\Gamma, \sigma)$ is said to be an extremal for \eqref{eq-1} if $f\neq 0$ a. e., and
\begin{equation}\label{eq-4}
\|\widehat{f\sigma}\|_{L^{2+4/d}(\R^{d+1})} = \mathcal{R}\|f\|_{L^2(\Gamma, \sigma)}.
\end{equation}
An extremising sequence for the inequality \eqref{eq-1} is a sequence $\{f_\nu\}\in L^2(\Gamma, \sigma)$ satisfies  $\|f_\nu\|_{L^2(\Gamma, \sigma)}=1$ and $\lim_{\nu\to\infty }\|\widehat{f_\nu\sigma}\|_{L^{2+4/d}(\R^{d+1})} =\mathcal{R}$.
An extremising sequence is said to be precompact if any subsequence has a sub-subsequence which is Cauchy in $L^2(\Gamma, \sigma)$.
\end{definition}

To prove the existence of extermizers for the Tomas-Stein inequality, we need to establish the compactness of a complex $L^2(\Gamma, \sigma)$ extremizing sequence. The noncompactness of the Tomas-Stein inequality are due to the modulation symmetry: $x\in S^d\mapsto e^{i x\cdot \xi}$ given $\xi\in \mathbb{R}^{d+1}$,  and the approximate scaling symmetry. The symmetry $f\mapsto f\cdot e^{ix\cdot \xi}$ merits further discussion and is closely related to the approximate symmetry. For the paraboloid $\mathbb{P}^d:= \{(y_1, \cdots, y_d, y_{d+1}): y_{d+1} = \frac 12(y_1^2+\cdots y_{d+1}^2)\}$, the analogy of the unimodular exponential $e^{ix\cdot \xi}$are quadratic exponentials $e^{i\eta\cdot x+it\tau |x|^2}$ with $(\eta, \tau)\in \mathbb{R}^{d+1}$, where $\xi, \eta\in \mathbb{R}^d$ ranges over a $d$-dimensional space. To see the analogy, consider a small neighborhood of $(0,\cdots, 0, 1) \in S^d$, equipped with coordinates $x'\in \mathbb{R}^d$ so that $x= (x',x_{d+1})$. Then for $\xi=(0,\cdots, 0, \lambda)$, $e^{ix\cdot \xi } = \exp\bigl(i\lambda (1-\frac 12 |x'|^2+O(|x'|^4))\bigr)$ for small $x'$; thus for small $x'$, one has essentially quadratic oscillation.  The presence of these symmetries among the extremizers for $\mathbb{P}^d$ implies that, in the language of the concentration-compactness theory \cite{Christ-Shao:extremal-for-sphere-restriction-I-existence, Christ-Shao:extremal-for-sphere-restriction-II-characterizations, kunze, Foschi:2007:maxi-strichartz-2d, Hundertmark-Zharnitsky:2006:maximizers-Strichartz-low-dimensions, Shao:2009}, an extremizer $f$ can be tight at a scale $r$, and $\hat{f}$ can simultaneously be tight at a scale $\hat{r}$, with the product $r\cdot \hat{r}$ arbitrarily large and thus loss of compactness is manifested, see Equation \eqref{eq-51}.

We define the optimal constant for the Strichartz inequality for the Schr\"odinger equations. 
\begin{equation}\label{Schrodinger-level}
\mathcal{R}_\mathbf{P} = (2\pi)^{-\frac 12 }\dfrac {\left( \int_{\mathbb{R}^{d+1}} |e^{it\Delta /2} \phi_G(x) |^{2+\frac 4d} dx dt \right)^{\frac {d}{2(d+2)}} }{\|\phi_G\|_{L^2(\mathbb{R})}}, \text{ where } \phi_G (x) = e^{-|x|^2/2}.
\end{equation}

\begin{theorem}\label{thm-existence} Assume that $\mathcal{R}>\mathcal{R}_{\mathbf{P}}$. Then there exists an extremal function $f\in L^2(\Gamma, \sigma)$ for \eqref{eq-1} by showing that any extremising sequence $\{f_\nu\}$ is precompact in $L^2(\Gamma, \sigma)$. This extremizer can be chosen to be a smooth function. 
\end{theorem}

\begin{remark}\label{Calderon-Zygmund}    
In \cite{Frank-Lieb-Sabin:2007:maxi-sphere-2d}, Frank, Lieb and Sabin assume that Gaussian functions are the extremizers optimizing \eqref{Schrodinger-level} and prove that $\mathcal{R}>\mathcal{R}_{\mathbf{P}}$ by a perturbative analysis. Similar analysis has previously appeared in \cite{Christ-Shao:extremal-for-sphere-restriction-I-existence, Christ-Shao:extremal-for-sphere-restriction-II-characterizations}.
\end{remark}

Let $\mathbb{P}^d$ be the paraboloid introduced above. Let $\sigma_{\mathbb{P}^d}$ be the measure $d\sigma_{\mathbb{P}^d} = dx_1\cdots dx_d$ on $\mathbb{P}^d$. The mapping $f\mapsto \widehat{f\sigma_{\mathbb{P}^d}}$ is likewise bounded from $L^2(\mathbb{P}^d, \sigma_{\mathbb{P}^d})$ to $L^{2+\frac 4d}_{t,x} (\mathbb{R}\times \mathbb{R}^d)$. Using $\mathcal{R}_{\mathbf{P}}$, 
\begin{align*}
e^{it\Delta/2}f(t,x)&:= \int_{\mathbb{R}^d } e^{ix\cdot y-i\frac {t|y|^2}{2}}f(y)dy.  \\
\|e^{it\Delta/2} f \|_{L^{2+\frac 4d}_{t,x}(\mathbb{R}\times \mathbb{R}^d)} &\le (2\pi)^{1/2}\mathcal{R}_{\mathbf{P}} \|f\|_{L^2(\mathbb{R}^d)}. 
\end{align*}
When $d=1,2$, Foschi \cite{Foschi:2007:maxi-strichartz-2d} has proved that extremals exist for this inequality, and moreover that every radial Gaussian is an extremal. Alternative proofs are given by Hundertmark and Zharnitisky \cite{Hundertmark-Zharnitsky:2006:maximizers-Strichartz-low-dimensions}  and Bennett, Bez, Carbery and Hundertmark \cite{Bennett-Bez-Carbery-Hundertmark: heat-flow}. The simple relation $\mathcal{R}\ge \mathcal{R}_{\mathbf{P}}$ is of significance. The relation follows from examination of a suitable sequence of trial functions $f_\nu$, such that $f_\nu(x)^2$ converges weakly to a Dirac mass on $S^d$, and $f_\nu$ is approximately a Gaussian in suitably rescaled coordinates depending on $\nu$. It is essential for this comparison that $\mathbb{P}^d$ has the same curvature at $0$ as $S^d$, which explains the factors of $\frac 12 $ in the definition of $\mathbb{P}^d$. 

Kunze \cite{kunze} is the first to discuss the existence of extremizers for the Strichartz/Fourier restriction inequalities and proved the existence of extremizers for the parabola in $\mathbb{R}^2$, and showed that any nonnegative extremizing sequence is precompact modulo the action of the natural symmetry group of the inequality. Several papers have subsequently dealt with related problems, in some cases determining all the extremizers explicitly \cite{Foschi:2007:maxi-strichartz-2d, Hundertmark-Zharnitsky:2006:maximizers-Strichartz-low-dimensions, Bennett-Bez-Carbery-Hundertmark: heat-flow}, in other cases merely proving existence \cite{Shao:2009}. A powerful result which leads easily in \cite{Shao:2009} to existence of extremizers is the profile decomposition \cite{Bahouri-Gerard:1999:profile-wave, Begout-Vargas:2007:profile-schrod-higher-d}. Of these works, the one closely related to ours is that of \cite{Bahouri-Gerard:1999:profile-wave, Begout-Vargas:2007:profile-schrod-higher-d, Shao:2009, Shao:2009profilesAiry}. One difficulty which we face is the lack of exact scaling symmetries. In some facets of the analysis this is merely a technical obstacle, but it is bound up with the most essential obstacle, which is the possibility that the optimal constant might be achieved only in a limit where $|f_\nu|^2$ tends to a Dirac mass on $\Gamma\subset S^d$. 

Our analysis follows the general profile decomposition framework in \cite{Christ-Shao:extremal-for-sphere-restriction-I-existence, Bahouri-Gerard:1999:profile-wave, Begout-Vargas:2007:profile-schrod-higher-d, Shao:2009, Shao:2009profilesAiry}. One earlier example is that the author proved the existence of extermizers for the one-dimensional Tomas-Stein inequality \cite{Shao:2016TS}; which has precursors in \cite{Christ-Shao:extremal-for-sphere-restriction-I-existence} and \cite{Christ-Shao:extremal-for-sphere-restriction-II-characterizations}. In the latter two works, Christ and the first author proved existence of extremals for the Tomas-Stein inequality for $S^2$, and obtained some exact characterizations of nonnegative extremals, complex extremals and complex extremising sequences. The general idea is the profile decomposition or the concentration-compactness approach \cite{lions1984a,lions1984b, lions1985a,lions1985b}. Roughly speaking, the profile decomposition starts with Bourgain's $X_p$-refinement of the Tomas-Stein inequality for the sphere and obtains a near-extremizer decomposition.  Each piece will be a $L^2$-normalized cap on the sphere. The second decomposition of this cap to upgrade the weak convergence to strong convergence is to display the approximate scaling symmetry, which resembles the Schr\"odinger behavior. The key to prove the orthogonality of profiles is to using the Tao bilinear restriction estimates for elliptic surfaces, see Proposition \ref{le-bilinear-sphere}. The novelty of our work is that we apply this estimate in the concentration-compactness theory of the maximizer problem for the Fourier restriction inequalities to establishing orthogonality of profiles. The analysis of the extremizer problem for the Tomas-Stein inequality for the two dimensional sphere in  \cite{Christ-Shao:extremal-for-sphere-restriction-I-existence} has extended to other manifolds \cite{Quilodran:2015, Carneiro2019},  and other norms (nonendpoint or mixed-norms) in \cite{Carneiro2019, COeS15, COeSS19}. The authors use mostly the techniques from calculus of variations, functional analysis and probability theory, though our main tools are, the central idea of Fourier Analysis and the refined linear restriction estimate (Lemma \ref{le-refinement-of-Tomas-Stein}) in the spirit of Bourgain \cite{Bourgain:1998:refined-Strichartz-NLS}, and the bilinear Fourier restriction estimates (Theorem \ref{thm-1}, Lemma \ref{le-bilinear-sphere}, and Theorem \ref{prop-full-decomp}) in the spirit of Wolff and Tao, \cite{Tao:2003:paraboloid-restri}. See the references therein. 

The current work aims to address the same question in \cite{Christ-Shao:extremal-for-sphere-restriction-I-existence} in higher dimensions. Note again that Frank, Lieb and Sabin answered this question in  \cite{Frank-Lieb-Sabin:2007:maxi-sphere-2d}, though our method of profile decompositions is different from theirs.  The first major difference is Frank, Lieb, Sabin uses the refined Strichartz inuequality, though we are using the refined Bourgain's Strichartz estimate in form of the $X_p$ spaces like what we did in  \cite{Christ-Shao:extremal-for-sphere-restriction-I-existence, Shao:2016TS}. The second main difference is that Frank, Lieb and Sabin uses the Brezis-Lieb lemma to prove the existence of maximizers, which they call it ``the method of missing mass" in \cite{Frank-Lieb-Sabin:2007:maxi-sphere-2d}. We are still using Tao's bilinear restriction estimate \cite{Tao:2003:paraboloid-restri} by approximating the sphere by paraboloids to prove orthogonality of profiles in the decomposition of the extremizing sequences (Lemma \ref{le-bilinear-sphere}, and Theorem \ref{prop-full-decomp}.) The reason we find Tao's bilinear restriction estimate for paraboloids fantastic because we use it to generalize the first author's previous work with Christ in establishing orthogonality of profiles in \cite[Lemma 7.5]{Christ-Shao:extremal-for-sphere-restriction-I-existence}, where we use the structure of convolution of the sphere surface measures. In the current work, Lemma \ref{le-bilinear-sphere}, we invoke the deep bilinear restriction theorem of Tao \cite{Tao:2003:paraboloid-restri}. The third difference is that we use the concentration-compactness idea, i.e., profile decompositions, which was developed in solving the mass-critical and energy-critical dispersive partial differential equations such as the Schr\"odinger equations, and wave equations, etc. See the reference in \cite{Tao:book}. 

In \cite{Christ-Shao:extremal-for-sphere-restriction-I-existence}, Christ and the first author have proved that constant functions are local extremizers to the Tomas-Stein inequality for $S^2$. The key ingredients are the use of spherical harmonics \cite{Stein-Weiss:1971:fourier-analysis, xu:2000:Funk_Hecke}. This idea was later used by Foschi \cite{Fo15} to establish that constants are global maximizers for this inequality.

The paper is divided into three parts: the first part focuses on establishing the refined Tomas-Stein estimates in the spirit of Bourgain in Section \ref{Archemedian-lever}, Lemma \ref{le-refinement-of-Tomas-Stein}. In Sections \ref{earthworm1}, \ref{earthworm2}, \ref{earthworm3}, \ref{earthworm4}, we establish the theorem of profile decompositions for a sequence of $L^2(\Gamma, \sigma)$-extremising sequence, Propositions \ref{prop-profiles} and \ref{le-fish-cute}. In Section \ref{Copernicus-meat}, we establish the conditional existence of extremizers for the Tomas-Stein inequality for $\Gamma\subset S^d$ in dimensions $d\ge 2$, Theorem \ref{thm-existence}. 

We will use the notations $X\lesssim Y$, $Y\gtrsim X$, or $X=O(Y)$ to denote the estimate $|X|\le
C Y$ for some constant $0<C<\infty$, which may depend on $d$ and the surface $\Gamma$ on the sphere, but not on
the functions. If $X\lesssim Y$ and $Y\lesssim X$ we will write $X\sim Y$. If the constant $C$
depends on a special parameter other than the above, we shall denote it explicitly by subscripts.
For example, $C_{\eps}$ should be understood as a positive constant not only depending on $d$
and the surface $\Gamma$, but also on $\eps$. The supports of functions on the sphere $S^d$ or $\Gamma\subset S^d$ can be understood from the context.

\section{The refined Tomas-Stein estimate}\label{Archemedian-lever}
\begin{definition}\label{def-cap}
The \emph{cap} $\mathcal{C}=\mathcal{C}(z,r)$ with center $z\in S^d$ and radius $r\in (0,1]$ is the set of all points $y\in S^d$ which lie in the same hemisphere as $z$ and are centered at $z$, and which satisfy $\pi_{H_z}(y)<r$, where the subspace $H_z\subset \R^{d+1}$ is the orthogonal complement of $z$ and $\pi_{H_z}$ denotes the orthogonal projection onto $H_z$.
\end{definition}

The following lemma on the refinement of the Tomas-Stein inequality is taken from \cite[Theroem 1.3]{Begout-Vargas:2007:profile-schrod-higher-d}.
\begin{lemma}\label{le-refinement-of-Tomas-Stein} 
For $f\in L^2(\Gamma, \sigma)$. There exists $\alpha\in (0,1)$ such that
\begin{equation}\label{eq-r1}
\|\widehat{f\sigma}\|_{2+4/d} \le C \left(\sup_\mathcal{C}\frac {1}{|\mathcal{C}|^{1/2}} \int_\mathcal{C} |f|d\sigma \right)^{\alpha} \|f\|^{1-\alpha}_{L^2(\Gamma, \sigma)},
\end{equation} where $\mathcal{C}$ denotes a cap on $S^d$.
\end{lemma}
To prove this lemma, we will use the Tao's blinear restriction estimate for paraboloids \cite{Tao:2003:paraboloid-restri}.  Recently Frank, Lieb and Sabin directly prove this estimate in \cite[Proposition 5.1]{Frank-Lieb-Sabin:2007:maxi-sphere-2d} using similar ideas.  We will provide a proof for completeness and follow the outline in Begout-Vargas \cite{Begout-Vargas:2007:profile-schrod-higher-d}. Namely we first prove a Bourgain's $X_p$ estimate. see e.g. \cite{Bourgain:1998:refined-Strichartz-NLS, Merle-Vega:1998:profile-schrod}. The $X_p$ estimate, Theorem \ref{thm-1}, for arbitrary dimensions is the first time in literature. We also give the second proof following the ideas in Killip and Visan \cite[Proposition 4.24]{Killip-Visan:2008:clay-lecture-notes}. We then deduce Lemma \ref{le-refinement-of-Tomas-Stein} from Theorem \ref{thm-1}.

For each integer $k\ge 0$ choose a maximal subset $ \{ z_k^j \} \subset S^d$ satisfying $|z_k^j -z_k^i| > 2^{-k} $  for all $i\neq j$. Then for any $x\in S^d$ there exists $z_k^i$ such that $|x-z_k^i| \le 2^{-k}$; otherwise $x$ could be adjoined to $z_k^i$ for this $z_k^i$, contradicting maximality. Therefore the caps $\mathcal{C}_k^j= (\mathcal{C}_k^j , 2^{-k+1})$ cover $S^d$ for each $k$, and there exists $C<\infty$ such that for any $k$, no point of $S^d$ belongs to more than $C$ of the caps $\mathcal{C}_k^j$. The constant $C$ is independent of $k$. This automatically give a decomposition for $\Gamma$.

For $p\in [1,\infty)$, the $X_p$ norm is defined by
\begin{equation}
\|f\|_{X_{p,q} }
=\left(  \sum_{k}\sum_j |\mathcal{C}_k^j|^{q/2} 
\big(
\frac{1}{|\mathcal{C}_k^j|} \int_{\mathcal{C}_k^j}|f|^p
\big)^{q/p}\right)^{1/q}.
\end{equation}

Moyua, Vargas, and Vega \cite{Moyua-Vargas-Vega:1999} have proved
\begin{mainthm}\label{moyua-vega-vargas}
There exist $0<C<\infty$ and $p\in (1,2)$
such that for any $f\in L^2(S^2)$,
\begin{equation}
\|{\widehat{f\sigma}}\|_{L^4(\mathbb{R}^{3})}
\le C\|f\|_{X_p}.
\end{equation}
\end{mainthm}
We generalize this lemma to arbitrary dimensions. 

\begin{theorem}\label{thm-1}
Let $q=\frac {2(d+2)}{d}$ and $p<2$ be such that $\frac {1}{p'} >\frac {d+3}{d+1} \frac 1q$. For every function $f\in X_{p,q}$, we have 
$$\|\widehat{f\sigma}\|_{L^q(\mathbb{R}^{d+1})}\le \|f\|_{X_{p,q}}. $$
\end{theorem}
We follow the outline in \cite{Begout-Vargas:2007:profile-schrod-higher-d} and give the details of the proof for the sake of completeness. Our main tool is the following sharp bilinear restriction estimate proved by Tao \cite[Remark 3]{Tao:2003:paraboloid-restri}. The precursor of this estimate, the sharp bilinear restriction estimate for the cone, is proved by Wolff \cite{Wolff:2001:restric-cone}. 

We write the upper quarter sphere as $\Gamma \subset S^d, \, \Gamma= \{(x, \phi(x)): \, |x|\le \frac 12, \phi(x) = \sqrt{1-|x|^2}\}$. For $f\in L^2(\Gamma, \sigma)$, 
$$\widehat{f\sigma} (t,x) = \int e^{ix\cdot \xi+ it\phi(\xi)} f(\xi, \sqrt{1-|\xi|^2}) \frac {d\xi}{\sqrt{1-|\xi|^2}} =  \int e^{ix\cdot \xi+ it\phi(\xi)}  \frac {f(\xi, \sqrt{1-|\xi|^2})}{\sqrt{1-|\xi|^2}} d\xi. $$
Since we are on the upper quarter sphere, the Jacobian factor $ \sqrt{1-|\xi|^2} \sim 1$; we regard $  \frac {f(\xi, \sqrt{1-|\xi|^2})}{\sqrt{1-|\xi|^2}}$ as a $d$-dimensional function on $\mathbb{R}^d$. We still denote it by $f$ in this section. 

\begin{mainthm}\cite{Tao:2003:paraboloid-restri} \label{thm-1-pro}
Let $Q, Q'$ be cubes of sidelength $1$ in $\mathbb{R}^d$ such that 
$$ \min\{d(x,y): \, x\in Q, y\in Q'\}\sim 1$$
and $f,g$ functions supported in $Q$ and $Q'$ respectively. Then for all $r>\frac {d+3}{d+1}$ and $p\ge 2$, we have 
$$ \|\widehat{f\sigma} \widehat{g\sigma}\|_{L^r(\mathbb{R}^{d+1})}  \le C \|f\|_{L^p(\mathbb{R}^d)} \|g\|_{L^p(\mathbb{R}^d)},$$
with a constant $C$ independent of $f,g, Q, Q'$. 
\end{mainthm}
By interpolation with the trivial estimate 
$$\|\widehat{f\sigma} \widehat{g\sigma} \|_{L^\infty(\mathbb{R}^{d+1})}  \le C \|f\|_{L^1}\|g\|_{L^1} \le C \|f\|_{L^p} \|g\|_{L^p}$$
for all $p\ge 1$, we can obtain
\begin{mainthm}\cite{Tao:2003:paraboloid-restri} \label{thm-2-pro}
Let $Q, Q'$ be cubes of sidelength $1$ in $\mathbb{R}^d$ such that 
$$ \min\{d(x,y): \, x\in Q, y\in Q'\}\sim 1$$
and $f,g$ functions supported in $Q$ and $Q'$ respectively. Then for all $r>\frac {d+3}{d+1}$ and $p$ such that $\frac {2}{p'}> \frac {d+3}{d+1} \frac 1r$, we have 
$$ \|\widehat{f\sigma} \widehat{g\sigma}\|_{L^r(\mathbb{R}^{d+1})}  \le C \|f\|_{L^p(\mathbb{R}^d)} \|g\|_{L^p(\mathbb{R}^d)},$$
with a constant $C$ independent of $f,g, Q, Q'$. 
\end{mainthm}
By a simple argument using Theorem \ref{thm-2-pro}, we obtain the following estimate. 
\begin{mainthm}\cite{Begout-Vargas:2007:profile-schrod-higher-d}\label{thm-3-pro}
Let $\tau, \tau'$ be cubes of sidelength $2^{-j}$ in $\mathbb{R}^d$ such that 
$$ \min\{d(x,y): \, x\in \tau, y\in \tau'\}\sim 2^{-j}$$
and $f,g$ functions supported in $\tau$ and $\tau'$ respectively. Then for all $r=\frac {d+2}{d}$ and $p$ such that $\frac {2}{p'}>\frac {d+3}{d+1}\frac 1r$, we have 
$$ \|\widehat{f\sigma} \widehat{g\sigma}\|_{L^r(\mathbb{R}^{d+1})}  \le C 2^{jd\frac {2-p}{p}} \|f\|_{L^p(\mathbb{R}^d)} \|g\|_{L^p(\mathbb{R}^d)},$$
with a constant $C$ independent of $f, g, \tau$ and $\tau'$. 
\end{mainthm}
Indeed, we choose one large cube $\tau_0\subset B(0,1)$ that has comparable sizes to $\tau, \tau'$ and contains $\tau, \tau'$. If necessary, we rotate $\tau_0$ such that the sides of $\tau_0$ are parallel to the coordinates. Neither $\widehat{f\sigma}$ nor $\|f\|_{L^p(\Gamma, \sigma)}$ is not affected by this rotation; the same applies to $\widehat{g\sigma}$ and $\|g\|_{L^p(\Gamma, \sigma)}$. This idea has also been implemented in \cite[Theorem A. 1]{Frank-Lieb-Sabin:2007:maxi-sphere-2d}.  Noting that the side length of $\tau_0$ is $2^{-j}$, by scaling,
$$ \widehat{f\sigma} (t,x) = 2^{-jd} \int_{Q} e^{i 2^{-j}x\cdot \eta +it \phi(2^{-j}\eta)} f(2^{-j}\eta) d\eta,$$
and 
$$ \widehat{g\sigma} (t,x) =2^{-jd} \int_{Q'}e^{i 2^{-j}x\cdot \eta +it \phi(2^{-j}\eta)} f(2^{-j}\eta) d\eta  $$
where $Q, Q'$ is the cube obtained from $\tau, \tau'$ under this scaling, respectively. Note that $Q, Q'$ are of distance $1$. Then we use Tao's sharp bilinear restriction estimate for elliptic surfaces, Theorem \ref{thm-2-pro}, from which Theorem \ref{thm-3-pro} follows. 

We will need to use the orthogonality of functions with disjoint supports. More precisely, the following Bernstein type lemma, a proof of which can be found, for instance, in Tao, Vargas and Vega \cite[Lemma 6.1]{Tao-Vargas-Vega:1998:bilinear-restri-kakeya}. 

\begin{mainthm}\label{Bernstein}
Let $\{R_k\}_{k\in \mathbb{N}}$ be a collection of rectangles in $\mathbb{R}^d$ and $c>0$, such that the dilates $(1+c) R_k$ are almost disjoint $i.e., \sum_k \chi_{(1+c)R_k} \le C$, and suppose that $(f_k)_{k\in \mathbb{N}}$ is a collection of functions whose Fourier transforms are supported in $R_k$. Then for all $1\le p\le \infty$, we have 
$$ \|\sum_{k\in \mathbb{N}} f_k \|_{L^p(\mathbb{R}^d)} \le C(d, c) \left(\|f_k\|^{p^*}_{L^p(\mathbb{R}^d)} \right)^{\frac{1}{p^*}},$$
where $p^*= \min(p, p')$.  
\end{mainthm}

Now we begin the proof of Theorem \ref{thm-1}, which we give two proofs. In the first proof, we follow the outline in Begout-Vargas \cite{Begout-Vargas:2007:profile-schrod-higher-d}. In the second proof, we follow the proof in Killip and Visan \cite[Proposition 4.24]{Killip-Visan:2008:clay-lecture-notes}.   Then we deduce Theorem \ref{le-refinement-of-Tomas-Stein} from it. 
\begin{proof}{\bf (The first proof.)}
We set $r=\frac q2=\frac {d+2}{d}$. We can assume that the support of $f$ is contained in the square $\{x\in \mathbb{R}^d:\, |x_i| \le \frac 12, 1\le i\le d\}$. For each $j\in \mathbb{Z}$, we decompose $\mathbb{R}^d$ into dyadic cubes $\tau_j^k$ of sidelength $2^{-j}$. Given a dyadic cube $\tau_j^k$ we will say that is the ``parent" of the $2^d$ dyadic cubes of sidelength $2^{-j-1}$ contained in it. We define that $\tau_k^j$ is ``orthogonal" to $\tau_{k'}^{j}$, denoted by $\tau_k^j\sim \tau_{k'}^j $,  if $\tau^j_k$ and $\tau^j_{k'}$ are not adjacent but have adjacent parents. For each $j\ge 0$, we write $f=\sum f_k^j$ where $f_k^j=f\chi_{\tau_k^j}$. Denote by $\beta$ the diagonal of $\mathbb{R}^d\times \mathbb{R}^d$, $\beta=\{ (x,x):\, x\in \mathbb{R}^d\}$. We have the following Whitney type decomposition, as indicated in Figure \ref{fig-1},  of $\mathbb{R}^d\times \mathbb{R}^d\setminus \beta$, 
$$[-1/2,1/2]^d\times [-1/2,1/2]^d \setminus \beta = \cup_j\cup_{k, k': \, \tau^j_k\sim \tau^j_{k'}} \tau_k^j\times \tau^j_{k'}. $$
 \begin{figure}[h]
        \centering
        \includegraphics[width=0.8\textwidth]{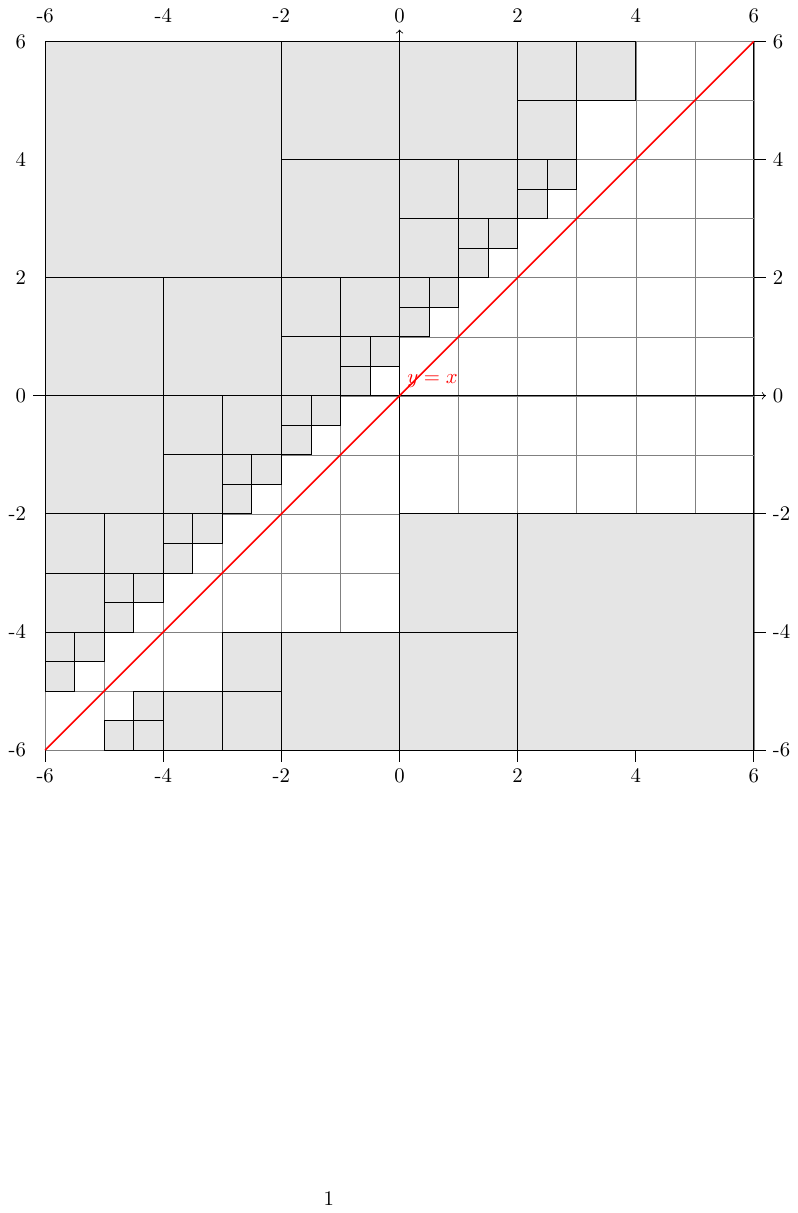}
        \caption{The Whitney type decomposition.}
        \label{fig-1}
    \end{figure}

This has previously appeared in \cite{Tao-Vargas-Vega:1998:bilinear-restri-kakeya}. Indeed, this Whitney type decomposition is possible because for given $a,b$, we consider the longest coordinate side-length $l_{max}$. There exists $k\in \mathbb{N}  $ such that $$2^{-k}<l_{max}\le 2^{-k+1}. $$
Now we do the binary divisions of each side. Then the two points $a,b$ locate inside different $2^{-k-1}$ cubes which are not adjacent but have adjacent parents. Note that these two little cubes are separated at least by a thin rectangle array of length $2^{-k-1}$ that is orthogonal to the longest side, and hence are not adjacent. Note that this is different from the usual Whitney decomposition, see e.g. \cite[Chapter 1, Theorem 3]{Stein:singular integrals}. 

So we can decompose $\widehat{f\sigma}$ correspondingly.  
\begin{align*}
\widehat{f\sigma}(t,x) \widehat{f\sigma} (t,x) &= \int_{\mathbb{R}^d} \int_{\mathbb{R}^d} e^{i(x\cdot\xi+t\phi(\xi))} e^{i(x\cdot \eta+t\phi(\eta))}f(\xi)g(\eta) d\xi d\eta \\
& = \sum_j\sum_k \sum_{k':\, \tau^j_k\sim \tau^j_{k'}} \widehat{f_k^j\sigma}(t,x) \widehat{f_{k'}^j \sigma}(t,x)
\end{align*}
The $(n+1)-$ dimensional Fourier support of $\widehat{f_k^j\sigma}$ is $\tilde{\tau_k^j} =\{(\phi(\xi), \xi):\, \xi \in \tau_k^j \}$. Hence the support of $\widehat{f_k^j\sigma} \widehat{f_{k'}^j \sigma}$ is contained in 
$$ \tilde{\tau_k^j}+\tilde{\tau_{k'}^j} =\{\bigl(\phi(\xi)+\phi(\xi'):\, \xi+\xi'\bigr):\, \xi\in \tau_k^j, \xi' \in \tau_{k'}^j \}. $$
Let $\phi$ be the elliptic phase defined in  Tao, Vargas, and Vega \cite[Page 968]{Tao-Vargas-Vega:1998:bilinear-restri-kakeya} and Tao \cite[Page 1381]{Tao:2003:paraboloid-restri}. Then 
\begin{align*}
 \phi(\xi)+ \phi(\xi') &=-\frac 12 |\xi|^2 -\frac 12 |\xi'|^2 +\epsilon \bigl(f(\xi)+f(\xi') \bigr)  \\
 &=\bigl( -\frac 12 +O(\epsilon) \bigr) \bigl( |\xi|^2+ |\xi'|^2\bigr) = \bigl( -\frac 12 +O(\epsilon) \bigr) \bigl( |\xi+\xi'|^2+ |\xi-\xi'|^2\bigr),
 \end{align*}
where $O(\epsilon)$ denotes a number that is between $0$ and $C\epsilon$ for some constant $C>0$. Indeed, the use of the big $O$ notation here is harmless because, after given $j, k$,  our goal is to counting the number of orthogonal pairs of cubes $\tau_{k'}^j$ such that $\tau_{k}^j\sim \tau_{k'}^j$ for a certain paraboloid in the sense of the definition at the beginning of this proof, see Inequalities \eqref{eq-orthogonal}  and \eqref{eq-orthogonal2} below.  We see that the Minkowski sum of the two sets above is contained in 
 $$ H_{j,k} = \{(b,a) \in \mathbb{R}\times \mathbb{R}^d:\, |a-\frac{k}{2^{j-1}} | \le \frac {C}{2^j},\frac 14 2^{-2j} \le b+(\frac 12 -O(\epsilon) ) |a|^2 \le 4d2^{-2j}, k\in \mathbb{Z}^{d}\},$$
For each $j, k$, $H_{j,k}$ is a the vertical truncation of an intercepted ``paraboloid'' with a cylinder in $\mathbb{R}^{d+1}$. Note that 
 \begin{equation}\label{eq-orthogonal}
 \sum_j \sum_k \sum_{k':\, \tau_k^j \sim \tau^j_{k'}} \chi_{H_{j,k}} \le C_d, 
 \end{equation}
 for some constant $C_d>0$. For each $j,k$, there are $C2^d$-$k'$ such that $\tau_{j, k'}\sim \tau_{j,k}$. Let $M=\max\{2[\ln (d+1)],2\}$. We decompose each $\tau_k^j$ into dyadic subcubes of sidelength $2^{-j-M}$. Consequently we have a corresponding decomposition of $\tau_k^j\times \tau_{k'}^j$ as follows. Set $\mathcal{D}$ the family of multi-indices $(m,m',l) \in \mathbb{Z}^d\times \mathbb{Z}^d\times \mathbb{Z}$, so that there exists $\tau_k^{l-M}$ and $\tau_{k'}^{l-M}$ with $\tau_m^l \subset \tau_k^{l-M}$, and $\tau^l_{m'} \subset \tau_{k'}^{l-M}$ and $\tau_k^{l-M}\sim \tau_{k'}^{l-M}$ where  $ j=l-M$. Then
 $$[-1/2,1/2]^d \times [-1/2,1/2]^d \setminus \beta = \cup_{\mathcal{D}} \tau_m^l \times \tau_{m'}^l. $$  
 Hence 
 $$\|\widehat{f\sigma}\|^2_{L^{2r} (\mathbb{R}^{d+1})}  = \|\widehat{f\sigma} \widehat{f\sigma}\|_{L^r(\mathbb{R}^{d+1})} = \|\sum_{\mathcal{D}} \widehat{f_m^l \sigma} \widehat{g_{m'}^l \sigma }\|_{L^r(\mathbb{R}^{d+1})}. $$
Note that if $(m,m',l) \in \mathcal{D}$, then the distance between $\tau_m^l$ and $\tau_{m'}^l$ is bigger than $2^{-l+M} \ge d2^{-l}$ and smaller than $4\sqrt{d} 2^{-l+M}$. We claim that there are rectangles $R_{m,m',l}$ and $c=c_d$ so that $\tau_m^l \times \tau_{m'}^l\subset R_{m,m',l}$ and $\sum_{\mathcal{D}} \chi_{(1+c) R_{m,m',l}} \le C_d$. This is the main reason we do a finer decomposition of $\tau_k^j$ taking into account that the ``paraboloid'' has a nonzero Gaussian curvature. We postpone the proof of this claim to the end of the proof. Assuming it holds, by Theorem \ref{Bernstein}, since $r<2$, we have 
$$ \|\sum_{\mathcal{D}} \widehat{f_m^l\sigma} \widehat{f_{m'}^l \sigma} \|_{L^r(\mathbb{R}^{d+1})} \le C_d \left( \sum_{\mathcal{D}} \|\widehat{f_m^l\sigma} \widehat{f_{m'}^l\sigma}\|^r_{L^r(\mathbb{R}^{d+1})} \right)^{1/r}. $$
Now we use Theorem \ref{thm-3-pro} to obtain
\begin{align*}
 & \left( \sum_{\mathcal{D}} \| \widehat{f_m^l\sigma} \widehat{f_{m'}^l\sigma}\|^r_{L^r(\mathbb{R}^{d+1})} \right)^{1/r} \\
 &\le C_{d, p} \left( \sum_l\sum_m \sum_{m':\, (m, m',l)\in \mathcal{D}} 2^{ldr\frac {2-p}{p}} \|f_m^l\|^r_{L^p(\mathbb{R}^d)} \|f_{m'}^l\|^r_{L^p(\mathbb{R}^d)} \right)^{1/r}
\end{align*}
Now for each $(m,l)$ there are at most $4^d2^{Md}$ indices $m'$ such that $(m,m',l) \in \mathcal{D}$. Hence 
\begin{align*}
 &\left( \sum_l\sum_m \sum_{m':\, (m, m',l)\in \mathcal{D}} 2^{ldr\frac {2-p}{p}} \|f_m^l\|^r_{L^p(\mathbb{R}^d)} \|f_{m'}^l\|^r_{L^p(\mathbb{R}^d)} \right)^{1/r} \\
 &\le C_d \left(\sum_l\sum_m 2^{ldr\frac {2-p}{p}} \|f_m^l\|^{2r}_{L^p(\mathbb{R}^d)}   \right)^{1/r}
\end{align*}
We will still have to justify the claim. Assume, for the sake of simplicity that $\tau_m^l\times \tau_{m'}^l$ is in the first quadrant of $\mathbb{R}^d$. Then $\tau_m^l+\tau_{m'}^l$ is contained on a set 
\begin{align*}
 H_{m, m',l}& :=\{ (a,b)\in \mathbb{R}^d\times \mathbb{R}:\, a =\frac {m+m'}{2^l}+\nu, \nu=(\nu_1, \cdots, \nu_n), 0\le \nu_i \le 2^{-l+1},\\
 & \frac 14 \times 2^{-2l+2M}\le b+\bigl(\frac 12 -O(\epsilon)\bigr) |a|^2 \le 16 d 2^{-2l+2M}\}. 
 \end{align*}
Consider the paraboloid defined by $$ b+\bigl(\frac 12 -O(\epsilon)\bigr) |a|^2 = \frac 14 \times 2^{-2l+2M}. $$ Take $\Pi_{m,m',l} $ to be the tangent plane to this paraboloid at the point $(a_0,b_0)$, with 
 $$a_0=(m+m')2^{-l}, b_0 = \bigl(-\frac 12 +O(\epsilon)\bigr) |a_0|^2 +\frac 14 \times 2^{-2l+2M}$$
 and passing through that point. Also consider the point $(a_1,b_1)$ with $a_1 = a_0 + (2^{-l+1}, \cdots, 2^{-l+1})$ and $b_1=\bigl(-\frac 12 +O(\epsilon)\bigr) |a_1|^2 -16 d \times 2^{-2l+2M}  $. Then the rectangle $R_{m,m',l}$ is defined as the only rectangle having a face contained in that hyperplane, and the points $(a_0,b_0)$, and $(a_1,b_1)$ as opposite vertices, as indicated in Figure \ref{fig-3}. 
 
   \begin{figure}[h]
        \centering
        \includegraphics[width=0.6\textwidth]{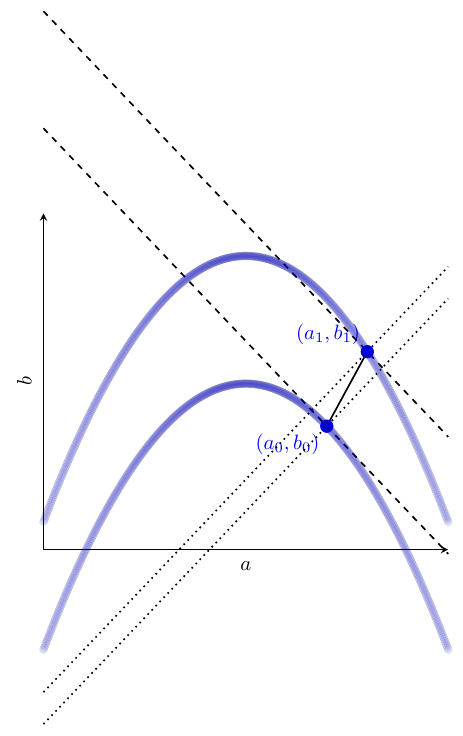}
        \caption{The refined Whitney type decomposition.}
        \label{fig-3}
    \end{figure}
 
Due to the convexity of paraboloids \cite[page 44]{Lieb-Loss:2001}, it follows that
$$H_{m,m',l}\subset R_{m,m',l}. $$
Moreover, one can see that, for small $c=c_d$ because of the consideration of the inscribed and circumscribed cubes to a sphere in $\mathbb{R}^d$, 
\begin{equation}\label{eq-v} 
\begin{split}
& (1+c)R_{m,m',l} \subset \{(a,b):\, a=(m+m')2^{-l}+\nu, \nu =(\nu_1, \cdots, \nu_d), |\nu_i |\le C_d 2^{-l+1},  \\
& C'_d 2^{-2l+2M}\le  b+\bigl(\frac 12 -O(\epsilon)\bigr) |a|^2 \le C''_d 2^{-2l+2M}  \}. 
\end{split}
\end{equation}
for some constants $C_d, C_d', C_d''$. Indeed, $C_d$ is natural. For $C_d'$ and $C_d''$,  the equation for the tangent plane $(a,b)\in \Pi_{m,m',l}$ is 
$$(1-O(\epsilon))a_0(a-a_0)+(b-b_0)=0. $$
Then $$|b-b_0| = \left| ( -1+O(\epsilon))a_0(a-a_0)\right|\le \frac {|m+m'|}{2^l} 2^{-l+1} \le 4d2^M 2^{-2l}\le C2^{-2l+2M}$$
for some absolute constant $C$, since the points are located in the square of side length $2^M$. Here $\frac M2\le \ln(d+1)\le M$ and so $d\sim 2^M. $ This is how the constants $C_d'$ and $C_d''$ are chosen. 

Therefore we have 
\begin{equation}\label{eq-orthogonal2}
\sum_{\mathcal{D}} \chi_{(1+c)R_{m,m',l}}=\sum_l\sum_m \sum_{m':\, (m, m',l)\in \mathcal{D} } \chi_{(1+c)R_{m,m',l}} \le C_d. 
\end{equation}
Indeed, we see that $H_{m,m',l}$ is contained in the rectangle $R_{m, m',l}$; and $(1+c) R_{m,m',l} $ is contained in a variant of the $H_{m,m',l}$ as in \eqref{eq-v}. Then given $m,l$, we count how many $m'$. First we know that there are $C_d$'s $k$. Then from the relation that $\tau_k^{l-M}\sim \tau_{k'}^{l-M}$, there are $C_d$'s $k'$. Then in $\tau_{k'}^{l-M}$, there are $C_d$'s $m'$. So in total there are $C_d$'s $m'$ if given $m, l$.  Now we finish the proof of Theorem \ref{thm-1}.
\end{proof}

\begin{proof}{\bf (The second proof. )}
We set $r=\frac q2=\frac {d+2}{d}$. We can assume that the support of $f$ is contained in the square $\{x\in \mathbb{R}^d:\, |x_i| \le \frac 12, 1\le i\le d\}$. For each $j\in \mathbb{Z}$, we decompose $\mathbb{R}^d$ into dyadic cubes $\tau_j^k$ of sidelength $2^{-j}$. Given a dyadic cube $\tau_j^k$ we will say that is the ``parent" of the $2^d$ dyadic cubes of sidelength $2^{-j-1}$ contained in it. We define that $\tau_k^j$ is ``orthogonal" to $\tau_{k'}^{j}$, denoted by $\tau_k^j\sim \tau_{k'}^j $,  if $\tau^j_k$ and $\tau^j_{k'}$ are not adjacent but have adjacent parents. For each $j\ge 0$, we write $f=\sum f_k^j$ where $f_k^j=f\chi_{\tau_k^j}$. Denote by $\beta$ the diagonal of $\mathbb{R}^d\times \mathbb{R}^d$, $\beta=\{ (x,x):\, x\in \mathbb{R}^d\}$. We have the following Whitney type decomposition, as indicated in Figure \ref{fig-1},  of $\mathbb{R}^d\times \mathbb{R}^d\setminus \beta$, 
$$[-1/2,1/2]^d\times [-1/2,1/2]^d \setminus \beta = \cup_j\cup_{k, k': \, \tau^j_k\sim \tau^j_{k'}} \tau_k^j\times \tau^j_{k'}. $$
This has previously appeared in \cite{Tao-Vargas-Vega:1998:bilinear-restri-kakeya} and  \cite[Proposition 4.24]{Killip-Visan:2008:clay-lecture-notes}. Indeed, given $\xi\neq \xi'$, there is a unique maximal pair $\xi\in \tau$ and $\xi'\in \tau'$ such that 
$$\tau\sim \tau', \text{ and }|\tau| =|\tau'|, \text{ and } \operatorname{dist} (\tau, \tau') \ge \frac {1}{\sqrt d}\operatorname{diam} (\tau). $$ 
Let $\mathcal{D}$ denote the collection of such pairs $\tau\sim \tau'$ as $\xi \neq \xi' $ varies over $[-\frac 12, \frac 12]^d. $ According to this definition, we have 
$$\sum_{(\tau, \tau') \in \mathcal{D}} \chi_\tau (\xi)\chi_{\tau'}(\xi') =1, \text{ for a.e. } (\xi, \xi')\in [-\frac 12, \frac 12]^d. $$
Note that since $\tau$ and $\tau'$ are maximal, $\operatorname{dist} (\tau, \tau') \le 4 \operatorname{diam}(\tau). $ In addition, this shows that given $\tau$,  there is a bounded number of $\tau'$ such that $(\tau, \tau')\in \mathcal{D}$. That is to say, 
$$\forall \tau,  \# \{\tau', \, \tau' \sim \tau \} \lesssim 1. $$
Note that this is different from the usual Whitney decomposition, see e.g. \cite[Chapter 1, Theorem 3]{Stein:singular integrals}. 

So we can decompose $\widehat{f\sigma}$ correspondingly.  
\begin{align*}
e^{-2it }\widehat{f\sigma}(t,x) \widehat{f\sigma} (t,x)&=\int_{\mathbb{R}^d} \int_{\mathbb{R}^d} e^{i(x\cdot\xi+t\phi(\xi))} e^{i(x\cdot \eta+t\phi(\eta))}f(\xi)g(\eta) d\xi d\eta \\
& =\sum_j\sum_k \sum_{k':\, \tau^j_k\sim \tau^j_{k'}} \widehat{f_k^j\sigma}(t,x) \widehat{f_{k'}^j \sigma}(t,x)
\end{align*}
The $(n+1)-$ dimensional Fourier support of $\widehat{f_k^j\sigma}$ is $\tilde{\tau_k^j} =\{(\phi(\xi), \xi):\, \xi \in \tau_k^j \}$. Hence the space-time Fourier transform support of $\widehat{f_k^j\sigma} \widehat{f_{k'}^j \sigma}$, $\operatorname{supp} \mathcal{F}_{t,x} \left( \widehat{f_k^j\sigma} \widehat{f_{k'}^j \sigma} \right) $,  is contained in 
$$ \operatorname{supp} \mathcal{F}_{t,x} \left( \widehat{f_k^j\sigma} \widehat{f_{k'}^j \sigma} \right)  \subset \tilde{\tau_k^j}+\tilde{\tau_{k'}^j} =\{\bigl(\phi(\xi)+\phi(\xi'):\, \xi+\xi'\bigr):\, \xi\in \tau_k^j, \xi' \in \tau_{k'}^j \}. $$
Here $\mathcal{F}_{t,x}(f)(\tau, \xi) =(2\pi)^{-\frac {d+1}{2}} \int_{\mathbb{R}^{d+1}} e^{i\xi \cdot x+it\tau} f(\tau, \xi) d\tau d\xi$.  The $\phi$ is the elliptic phase defined in  Tao, Vargas, and Vega \cite[Page 968]{Tao-Vargas-Vega:1998:bilinear-restri-kakeya} and Tao \cite[Page 1381]{Tao:2003:paraboloid-restri}. We follow the approach outlined in \cite[Lemma A.2]{Frank-Lieb-Sabin:2007:maxi-sphere-2d} or \cite[Proposition 4.24]{Killip-Visan:2008:clay-lecture-notes}. The Taylor expansion leads to the formula
\begin{align*}
\phi(\xi)+\phi(\xi') &= 2\phi(\frac {\xi+\xi'}{2}) +a(\xi, \xi') |\xi-\xi'|^2, \\
\phi(\frac {\xi+\xi'}{2}) &= \phi(\frac {c(\tau_k^j+\tau_{k'}^j)}{2}) +\frac 12 \nabla \phi (\frac {c(\tau_k^j+\tau_{k'}^j)}{2}) \cdot (\xi+\xi'-c(\tau_k^j+\tau_{k'}^j))\\
&+b(\eta+\eta', c(\tau_k^j+\tau_{k'}^j)) |\xi+\xi'-c(\tau_k^j+\tau_{k'}^j)|^2,
\end{align*}
where $c(\tau_k^j+\tau_{k'}^j)$ denotes the center of the cube $\tau_k^j+\tau_{k'}^j$, for two functions $a$ and $b$ satisfying 
$$ -\frac 38 \le a(\xi, \xi')\le -\frac 18, \, -\frac {3}{16} \le b(\xi+\xi', c(\tau_k^j+\tau_{k'}^j))  \le -\frac {1}{16},$$
assuming that $\epsilon$ in defining the elliptic phase in Tao, Vargas, and Vega \cite[Page 968]{Tao-Vargas-Vega:1998:bilinear-restri-kakeya} and Tao \cite[Page 1381]{Tao:2003:paraboloid-restri} is sufficiently small. Further the space-time Fourier transform
$$ \operatorname{supp} \mathcal{F}_{t,x} \left( \widehat{f_k^j\sigma} \widehat{f_{k'}^j \sigma} \right) \subset R(\tau_k^j+ \tau_{k'}^j),$$
where 
$$ R(\tau_k^j+ \tau_{k'}^j) := \left\{ (b,a):\, a \in \tau_k^j+ \tau_{k'}^j, \, -2\le \frac { b-2\phi(\frac {c(\tau_k^j+\tau_{k'}^j)}{2})- \nabla \phi (\frac {c(\tau_k^j+\tau_{k'}^j)}{2}) \cdot (a-c(\tau_k^j+\tau_{k'}^j)) }{\operatorname{diam}(\tau_k^j+\tau_{k'}^j)^2} \le -\frac {1}{32} \right\}. $$
Note that in computing these bounds forming the set, we have used 
$$ 0\le |a-c(\tau_k^j+\tau_{k'}^j)| \le \operatorname{diam}(\tau_k^j), \, \operatorname{diam}(\tau_k^j)\le |\xi-\xi'|\le 4 \operatorname{diam}(\tau_k^j), \operatorname{diam}(\tau_k^j+\tau_{k'}^j) = 2 \operatorname{diam}(\tau_k^j). $$
Again by the Taylor expansion, we have 
$$- \frac {3}{16} |c(\tau_k^j+\tau_{k'}^j) -a|^2 \le  \phi(\frac {a}{2}) - \phi(\frac {c(\tau_k^j+\tau_{k'}^j)}{2}) -\nabla \phi (\frac {c(\tau_k^j+\tau_{k'}^j)}{2})\cdot \frac 12 (a-c(\tau_k^j+\tau_{k}^j))\le  -\frac 1{16} |c(\tau_k^j+\tau_{k'}^j ) - a|^2. $$
Therefore for $(b,a) \in R(\tau_k^j+\tau_{k'}^j )$, noting that $0\le |c(\tau_k^j+\tau_{k'}^j) -a|\le \operatorname{diam} (\tau_k^j)$, 
\begin{equation*}
\begin{split}
b-2\phi(\frac a2) &= -2\left( \phi(\frac a2 ) -\phi(\frac {c(\tau_k^j+\tau_{k'}^j)}{2})  -\nabla \phi (\frac {c(\tau_k^j+\tau_{k'}^j)}{2}) \cdot (\frac {a-c(\tau_k^j+\tau_{k'}^j)}{2}) \right)\\
&+b +2\left(-\phi(\frac {c(\tau_k^j+\tau_{k'}^j)}{2})  -\nabla \phi (\frac {c(\tau_k^j+\tau_{k'}^j)}{2}) \cdot (\frac {a-c(\tau_k^j+\tau_{k'}^j)}{2})\right). 
\end{split}
\end{equation*}
Then $$ -8 \operatorname{diam}(\tau_k^j)^2 +\frac 18 |c(\tau_k^j+\tau_{k'}^j)-a|^2\le b-2\phi(\frac a2) \le -\frac 18 \operatorname{diam}(\tau_k^j)^2 +\frac 38 |c(\tau_k^j+\tau_{k'}^j)-a|^2. $$
That is to say, 
$$ -8 \operatorname{diam}(\tau_j^k)^2 \le b-2\phi(\frac a2)  \le -\frac 1{32} \operatorname{diam}(\tau_k^j)^2. $$
Therefore if two pair of close cubes $\tau_k^j\sim \tau_{k'}^j$ and $\tau_{m}^n\sim \tau_{m'}^n$ are such that $R(\tau_k^j+\tau_{k'}^j)$ and $R(\tau_m^n+\tau_{m'}^n)$ intersect, they must have a similar diameter. The same holds for the dilates $(1+\alpha) R(\tau_k^j+\tau_{k'}^j)$ for some small $\alpha$ by the same argument. Indeed, we can take $\alpha <0.01$. Lastly if the diameters are in a finite number, the cubes are also in a finite number since their centers satisfy
$$|c(\tau_k^j+\tau_{k'}^j) - c(\tau_m^n+\tau_{m'}^n)| \le |c(\tau_k^j+\tau_{k'}^j) -a| +| c(\tau_m^n+\tau_{m'}^n) -a| \le 2 \operatorname{diam}(\tau_j^k). $$
Then we see that 
 \begin{equation}\label{eq-orthogonal3}
 \sup_{(b,a)} \sum_{(\tau_k^j, \tau_{k'}^j)\in \mathcal{D}} \chi_{(1+\alpha) R(c(\tau_k^j+\tau_{k'}^j)}(b,a)\le C_d, 
 \end{equation}
 for some constant $C_d>0$. Then
  $$\|\widehat{f\sigma}\|^2_{L^{2r} (\mathbb{R}^{d+1})}  = \|\widehat{f\sigma} \widehat{f\sigma}\|_{L^r(\mathbb{R}^{d+1})} = \|\sum_{\mathcal{D}} \widehat{f_k^j \sigma} \widehat{g_{k'}^j \sigma }\|_{L^r(\mathbb{R}^{d+1})}. $$
By Theorem \ref{Bernstein}, since $r<2$, we have 
$$ \|\sum_{\mathcal{D}} \widehat{f_k^j\sigma} \widehat{f_{k'}^j \sigma} \|_{L^r(\mathbb{R}^{d+1})} \le C_d \left( \sum_{\mathcal{D}} \|\widehat{f_k^j\sigma} \widehat{f_{k'}^j\sigma}\|^r_{L^r(\mathbb{R}^{d+1})} \right)^{1/r}. $$
Now we use Theorem \ref{thm-3-pro} to obtain
\begin{align*}
 & \left( \sum_{\mathcal{D}} \| \widehat{f_k^j\sigma} \widehat{f_{k'}^j\sigma}\|^r_{L^r(\mathbb{R}^{d+1})} \right)^{1/r} \\
 &\le C_{d, p} \left( \sum_j\sum_k \sum_{k'} 2^{jdr\frac {2-p}{p}} \|f_j^k\|^r_{L^p(\mathbb{R}^d)} \|f_{k'}^j\|^r_{L^p(\mathbb{R}^d)} \right)^{1/r},
\end{align*}
where $p,q,r$ are assumed in Theorem \ref{thm-3-pro}.  Now for each $(j,k)$ there are at most $2d$ indices $k'$ such that $\tau_k^j \sim \tau_{k'}^j$. Hence 
\begin{align*}
 &\left( \sum_j \sum_k \sum_{k'} 2^{ldr\frac {2-p}{p}} \|f_k^j\|^r_{L^p(\mathbb{R}^d)} \|f_{k'}^j\|^r_{L^p(\mathbb{R}^d)} \right)^{1/r} \\
 &\le C_d \left(\sum_j\sum_l 2^{jdr\frac {2-p}{p}} \|f_k^j\|^{2r}_{L^p(\mathbb{R}^d)}   \right)^{1/r}
\end{align*}
The last sum is recognized to be $\|f\|_{X_{p,q}}$ with $q=2r=\frac {2(d+2)}{d}$ and $\frac {1}{p'}>\frac {d+3}{d+1}\frac 1q$. Thus
$$\|\widehat{f\sigma}\|_{L^q(\mathbb{R}^{d+1})}\le \|f\|_{X_{p,q}}. $$
 Now we finish the proof of Theorem \ref{thm-1}.
\end{proof}

By using this Bourgain’s $X_{p,q}$ estimate, we will prove Lemma \ref{le-refinement-of-Tomas-Stein}. 
\begin{lemma}\label{le-intermediate}
For $p\in (1, 2)$. Then there exists $\mu\in (0,\frac 1p)$ such that for any $f\in L^2(\Gamma, \sigma)$, we have 
\begin{equation}
\| \widehat{f\sigma}\|_{L^{2+\frac 4d}}
\le C \left(  \sup_{(j,k) \in \mathbb{Z}\times \mathbb{Z}^d}  |\mathcal{C}_k^j|^{\frac {p-2}{2}}  \int_{\mathcal{C}_k^j}|f|^p \right)^{\mu} \|f\|^{1-\mu p}_{L^2( \Gamma, \sigma)}.
\end{equation}
where $C=C(p,q)$ and $\mu=\mu(p,q)$. 
\end{lemma} 
If this lemma were established, then Lemma \ref{le-refinement-of-Tomas-Stein} follows by a simple application of H\"older inequality. For its proof, we will follow \cite[Theorem 1.3]{Begout-Vargas:2007:profile-schrod-higher-d}. 
\begin{proof}
{\bf Step 1. } By homogeneity, we can assume that $\|f\|_{L^2(\Gamma, \sigma)}=1$. Then it suffices to show that for any function $f\in L^2(S^d)$ such that $\|f\|_{L^2(\Gamma, \sigma)} =1$, 
$$ \sum_{k}\sum_j |\mathcal{C}_k^j|^{q/2} 
\big(\frac{1}{|\mathcal{C}_k^j|} \int_{\mathcal{C}_k^j}|f|^p
\big)^{q/p} \le C(p,q) \left(\sup_{j,k} \left( |\mathcal{C}_k^j|^{\frac {p-2}{2p}} \bigl( \int_{\mathcal{C}_k^j} |f|^p\bigr)^{1/p} \right) \right)^{\alpha}, $$
where $\alpha = \mu pq$ and where $\mu$ has to be determined. We take $\alpha$ and $\beta$ such that $\frac 2q<\beta<1$, $\beta>\frac p2$ and $\alpha+q\beta =q$. Such $\alpha,\beta$ are possible if $\mu$ is chosen such that $$0<\mu<\min\{\frac 1p-\frac 12, \, \frac {2}{(d+2)p}\}. $$ Then taking out the $L^\infty$ norm directly, 
\begin{equation*}
\begin{split}
& \sum_{k}\sum_j |\mathcal{C}_k^j|^{q/2} 
\big(\frac{1}{|\mathcal{C}_k^j|} \int_{\mathcal{C}_k^j}|f|^p\big)^{q/p} =\sum_{k}\sum_j |\mathcal{C}_k^j|^{\frac {p-2}{2p} q} 
\big(\int_{\mathcal{C}_k^j}|f|^p\big)^{q/p} \\
&\le C(p,q) \left(\sum_{j,k} |\mathcal{C}_k^j|^{\frac {p-2}{2p} \beta q} \left( \int_{\mathcal{C}_k^j } |f|^p \right)^{\frac {\beta q} p}    \right) \left(\sup_{j,k} \left( |\mathcal{C}_k^j|^{\frac {p-2}{2p}} \bigl( \int_{\mathcal{C}_k^j} |f|^p\bigr)^{1/p} \right) \right)^{\alpha}
\end{split}
\end{equation*}
We set $\mu=\frac {\alpha}{pq} =\frac {1-\beta}{p} \in (0,\frac 1p)$. Hence it is enough to show that 
$$ \sum_{j,k} |\mathcal{C}_k^j|^{\frac {p-2}{2p} \beta q} \left( \int_{\mathcal{C}_k^j } |f|^p \right)^{\frac {\beta q} p} \le C(p,q). $$
{\bf Step 2. }Since $\beta q >2$, $\frac {\beta q} p>1$.  By the convexity of the $\ell^{\frac {\beta q}{p}}$ with $\frac {\beta q}{p}>1$, we split the sum 
\begin{equation*}
\begin{split}
& \sum_{j,k} |\mathcal{C}_k^j|^{\frac {p-2}{2p} \beta q} \left( \int_{\mathcal{C}_k^j } |f|^p \right)^{\frac {\beta q} p} \\
&\le C \sum_{j,k} |\mathcal{C}_k^j|^{\frac {p-2}{2p} \beta q} \left( \int_{\mathcal{C}_k^j \cap \{|f|>|\mathcal{C}_k^j |^{-1/2}\} } |f|^p \right)^{\frac {\beta q} p} \\
&+ C \sum_{j,k}  |\mathcal{C}_k^j|^{\frac {p-2}{2p} \beta q} \left( \int_{\mathcal{C}_k^j \cap \{|f|\le |\mathcal{C}_k^j|^{-1/2}\} } |f|^p \right)^{\frac {\beta q} p}\\
&:= C(A+B),
\end{split}
\end{equation*}
where $C=C(p,q)$. We estimate the first term. For $j\in \mathbb{Z}$, define $$f^j = f1_{\{|f|>|\mathcal{C}_k^j |^{-1/2}\} }. $$
Then 
$$A= \sum_{j,k} |\mathcal{C}_k^j|^{\frac {p-2}{2} \frac {\beta q}{p}} \left( \int_{\mathcal{C}_k^j } |f^j|^p \right)^{\frac {\beta q} p}. $$
Then by the comparison of the $\ell^{\frac {q \beta }{p}}$ norm and the Fubini theorem, we have 
\begin{equation*}
\begin{split}
A &\le \left(  \sum_{j,k}  |\mathcal{C}_k^j|^{\frac {p-2}{2}}  \int_{\mathcal{C}_k^j} |f^j |^p \right)^{\frac {\beta q} p}\\
&=  \left(  \sum_j |\mathcal{C}_k^j|^{\frac {p-2}{2}}  \int  |f^j |^p \right)^{\frac {\beta q} p}  \\
&\le \left( \int_{S^d} |f|^p \sum_{j:\, |f| > |\mathcal{C}_k^j |^{-1/2} } |\mathcal{C}_k^j |^{\frac {p-2}{2}} \right)^{\frac {\beta q}{p}} 
\end{split}
\end{equation*}
Since $2-p>0$, we can sum the geometric series and obtain 
$$A\le C \left(\int_{S^d} |f|^p |f|^{2-p} \right)^{\frac {\beta q}p } \le C \left(\int_{S^d} |f|^2 \right)^{\frac {\beta q}p } \le C$$
because $\|f\|_{L^2(\Gamma, \sigma)} =1$. We now estimate the term $B$. For $j\in \mathbb{Z}$, define $f_j = f1_{|f|\le |\mathcal{C}_k^j |^{-1/2}}$. Then
$$B=  \sum_{j,k} |\mathcal{C}_k^j|^{\frac {p-2}{2p} \beta q} \left( \int_{\mathcal{C}_k^j } |f_j|^p \right)^{\frac {\beta q} p}.$$
We use the H\"older inequality with exponents $\frac {\beta q}{p}$ and $\frac {\beta q}{\beta q-p}$ to the term $\int_{\mathcal{C}_k^j} |f_j|^p$: 
\begin{equation*}
\begin{split}
B & \le \sum_{j,k} |\mathcal{C}_k^j|^{\frac {p-2}{2p} \beta q}  \int_{\mathcal{C}_k^j } |f_j|^{\beta q} \left( |\mathcal{C}_k^j|^{\frac {\beta q -p}{\beta q}}\right)^{\frac {\beta q}p} \\
&=  \sum_{j,k} |\mathcal{C}_k^j|^{\frac{\beta q}{2}-1}  \int_{\mathcal{C}_k^j } |f_j|^{\beta q} \\
&=\sum_{j, k} 2^{jd (1-\frac {\beta q}2)} \int_{\mathcal{C}_k^j } |f_j|^{\beta q} \le \sum_{j} 2^{jd (1-\frac {\beta q}2)} \int_{S^d} |f_j|^{\beta q}  \\
&= \int_{S^d} |f|^{\beta q} \sum_{j:\, |f| \le 2^{jd/2}} 2^{jd(1-\frac {\beta q}{2})} \le  \int_{S^d} |f|^2\le 1, 
\end{split}
\end{equation*}
since $1-\frac {\beta q}{2} <0$ and we are summing a geometric series. Then Lemma \ref{le-intermediate} follows. 
\end{proof}

\section{Profile decomposition: the first decompostion}\label{earthworm1}
We first state several definitions as in \cite{Christ-Shao:extremal-for-sphere-restriction-I-existence}.
\begin{definition}\label{rescaled-maps}
Let $\mathcal{C} = \mathcal{C}(z,r)$ be a cap. For $z\in S^d$, define 
$$ \Psi_z(x) =r^{-1} L_z(\Pi_{H_z}(x))$$
for all $x$ in the hemisphere $\{x:\, x\cdot z \ge 0\}$, where $\Pi_{H_z}$ is the orthogonal projection onto $H_z$, and $L_z: H_z\to \mathbb{R}^2$ is an arbitrary linear isometry. The rescaled map associated with $\mathcal{C}$ is defined by $\Phi_{\mathcal{C}_{z,r}} =\Psi_z^{-1}$. 
\end{definition}

\begin{remark}
The map $\Phi_{\mathcal{C}(z,r)}: \, B(0,r^{-1}) \to \text{`` the indicated hemisphere"}$. For instance, for $z=(0,0,1)$, 
$$ \Phi_{\mathcal{C}(z,r)} (y_1,y_2) = (ry_1,ry_2, (1-r^2|y|^2)^{1/2}), \forall y \in B(0,r^{-1}). $$
\end{remark}

\begin{definition}\label{normalized-scaling}
Let $\mathcal{C} =\mathcal{C}(z,r)$ be a cap. For $f\in L^2(S^d)$, define the pullback of $f$ by 
$$ \Phi_{\mathcal{C}}^* f(y) = r^{d/2} \bigl( f\circ \Phi_\mathcal{C} \bigr) (y). $$
\end{definition}

\begin{remark} 
We compute the following $L^2$ norms and find that $\| \Phi_\mathcal{C}^* f\|^2_{L^2(\mathbb{R}^d)}$ and $\|f\|_{L^2(\Gamma, \sigma)}$ are comparable. 
\begin{equation}\label{sphere-Jacobian}
\begin{split}
\| \Phi_\mathcal{C}^* f\|^2_{L^2(\mathbb{R}^d)} &= r^d \|f\circ \Phi_\mathcal{C}\|_{L^2(\mathbb{R}^d)}  \\
& =r^d \int_{B(0,r^{-1})} \left| f\circ \Pi^{-1}_{H_z}\circ L_z^{-1}(ry) \right|^2 dy \\
& =r^d \int_{B(0,r^{-1})} \left| f\circ \Pi^{-1}_{H_z}\circ L_z^{-1}(ry) \right|^2 \sqrt{1-|ry|^2} \frac {dy}{\sqrt{1-|ry|^2}}  \\
& = r^d \int_{S^d} \left| f(z)\right|^2 J_y d\sigma(z)\\
\end{split}
\end{equation}
Here $J_y = \sqrt{1-r^2|y|^2}$. This formula and the last two lines in \eqref{sphere-Jacobian} use \cite[Eq. (68) on page 498]{Stein:1993}. We see that if $0<r<1/2$, then $\frac {\sqrt{3}} {2} \le J_y \le 1$. 
\end{remark}

As a consequence of the refinement in Lemma \ref{le-refinement-of-Tomas-Stein}, we have
\begin{proposition}\label{prop-first-decomp}
For any $\delta>0$ there exists $C_\delta<\infty$ and $\eta_\delta>0$ with the following properties. If $f_\nu\in L^2(\Gamma, \sigma)$ satisfies $\|\widehat{f_\nu\sigma}\|_{2+4/d} \ge \delta \mathcal{R} \|f_\nu\|_2$, then there exists $N\in \mathbb{N}$,  $\{f^j_\nu\}$, $\mathcal{C}_\nu^j$ and $e_\nu^N$ satisfying that
\begin{align}
& f_\nu=\sum_{j=1}^N f_\nu^j+e_\nu^N,\\
& \limsup_{\nu} \frac {r_\nu^j}{r_\nu^k}+\frac {r_\nu^k}{r_\nu^j}+\frac {|z_\nu^j-z_\nu^k|}{r_\nu^j}=\infty,\,\|f_\nu\|_2^2=\sum_{j=1}^N \|f_\nu^j\|_2^2+ \|e_\nu^N\|_2^2,\\
& |f_\nu^j|\le C_\delta \|f_\nu\|_2|\mathcal{C}_\nu^j|^{-1/2} \chi_{\mathcal{C}_\nu^j},\,\|f_\nu^j\|_2\ge \eta_\delta \|f_\nu\|_2,\\
& \|\widehat{e^N\sigma}\|_{2+4/d} \le \delta \mathcal{R} \|f_\nu\|_2.
\end{align} Here both $C^{-1}_\delta$ and $\eta_\delta$ are proportional to $\delta^{O(1)}$. 
\end{proposition}
The proof uses the refined Strichartz estimate, Lemma \ref{le-refinement-of-Tomas-Stein} and iterations. During the iterations, the orthogonality relation is extracted along subsequences. 
\begin{proof}
{\bf Step 1. }Let $f_\nu$ be such that $\|f_\nu \|_{L^2(\Gamma, \sigma)}=1$ and $0<\delta<1$. Then from Lemma \ref{le-refinement-of-Tomas-Stein}, we see that 
$$ \left( \sup_{\mathcal{C}_\nu} |\mathcal{C}_\nu|^{-1/2} \int_{\mathcal{C}_\nu} |f_\nu| d\sigma\right)^{\alpha} \ge c \delta,$$
for some absolute constant $0<c$. That is to say, 
$$  \sup_{\mathcal{C}_\nu} |\mathcal{C}_\nu|^{-1/2} \int_{\mathcal{C}_\nu} |f_\nu| d\sigma \ge c\delta^{1/\alpha}. $$
Fix a cap $\mathcal{C}_\nu$ such that $$  |\mathcal{C}_\nu|^{-1/2} \int_{\mathcal{C}_\nu} |f_\nu| d\sigma \ge c_0 \delta^{1 /\alpha},$$
for some constant $c_0=\frac c2$. 
That is to say, 
$$  \int_{\mathcal{C}_\nu} |f_\nu| d\sigma \ge c_0 \delta^{1 /\alpha} |\mathcal{C}_\nu|^{1/2}.$$
Let $R_\nu \ge 1$. Define $E_\nu=\{ x\in \mathcal{C}_\nu: |f_\nu(x)|\le R_\nu\}$. Set $g_\nu=f_\nu1_{E_\nu}$ and $h_\nu=f_\nu-f_\nu 1_{E_\nu}$. Then $g_\nu$ and $h_\nu$ have disjoint supports, $g_\nu+h_\nu=f_\nu$, $g_\nu$ is supported on $\mathcal{C}_\nu$, and $\|g_\nu\|_\infty \le R_\nu$. 
Now $|h_\nu(x)| \ge R_\nu$ for almost every $x\in \mathcal{C}_\nu$ for which $h_\nu(x) \neq 0$, so 
$$ \int_{\mathcal{C}_\nu} |h_\nu| \le R_\nu^{-1} \int_{\mathcal{C}_\nu }|h_\nu|^2 \le {R_\nu}^{-1} \|f_\nu\|_2^2 = R_\nu^{-1}. $$
Define $R_\nu$ by $R_\nu^{-1} =\frac {c_0}2 \delta^{1/\alpha} |\mathcal{C}_\nu|^{1/2}$. Then 
$$ \int_{ \mathcal{C}_\nu }|g_\nu| =\int_{ \mathcal{C}_\nu }|f_\nu| -\int_{ \mathcal{C}_\nu }|h_\nu| \ge \frac {c_0}2 \delta^{1/\alpha} |\mathcal{C}_\nu|^{1/2}. $$
By H\"older's inequality, since $g_\nu$ is supported on $\mathcal{C}_\nu$, we have 
$$ \|g_\nu\|_2 \ge |\mathcal{C}_\nu|^{-1/2} \|g_\nu\|_{L^1(\mathcal{C}_\nu)} \ge \frac {c_0}{2}\delta^{1/\alpha} = \frac {c_0}{2}\delta^{a/\alpha} \|f_\nu\|_{L^2}. $$
Then the decomposition $f_\nu =g_\nu+h_\nu$ satisfies the conclusion that, for every $\delta>0$, there exists $C_\delta>0$ and $\eta_\delta >0$, such that for every $\delta$-nearly extremizer, there exists caps $\mathcal{C}_\nu$ such that 
\begin{equation}
\begin{split}
& 0\le |g_\nu|, |h_\nu| \le |f_\nu|,\\
& \text{ the functions } g, h \text{ have disjoint supports}, \\
&  |g_\nu(x)| \le C_\delta \|f_\nu\|_{L^2} |\mathcal{C}_\nu|^{-1/2} 1_{\mathcal{C}_\nu}(x), \\
& \|g_\nu\|_{L^2}\ge \eta_\delta \|f_\nu\|_{L^2}.  
\end{split}
\end{equation}

{\bf Step 2.} It is easy to see that
\begin{equation}\label{uniform-bound}
 \Phi_{\mathcal{C}_{\nu}}^*(g_\nu) (y) =r_\nu^{d/2} \left( g_\nu \circ \Phi_{\mathcal{C}_\nu }\right) (y) \le C_\delta 1_{B(0,1)}(y), \forall y \in \mathbb{R}^d. 
 \end{equation}
 Also we have
 \begin{equation}\label{L2iteration}
 \|f_\nu\|^2_{L^2(\Gamma, \sigma)} = \|f_\nu-g_\nu\|^2_{L^2(\Gamma, \sigma)} +\|g_\nu \|^2_{L^2(\Gamma, \sigma)}
 \end{equation}
 since the supports are disjoint from the sphere $S^d$. We repeat the same arguments with $f_\nu-g_\nu$ in place of $g_\nu$. At each step, the $L^2$-norm decreases of at least $c\delta^{1/\alpha}$. With the same constant $c$ as for the first step. After $N(\delta)$ steps, we obtain $(g_\nu^j)_{1\le j\le N(\delta)}$ and $( \gamma_\nu^j)_{1\le j\le N(\delta)} = (z_\nu^j, r_\nu^j) $ satisfying the relation \eqref{uniform-bound} 
 \begin{equation}\label{iteration2}
 \begin{split}
 f_\nu & = \sum_{j=1}^{N(\delta)} g_\nu^j +q_\nu, \\
 \|\widehat{q_\nu \sigma }\|_{L^{2+\frac 4d} (\mathbb{R}^{d+1})} &\le \delta, \\
  \|f_\nu\|^2_{L^2(\Gamma, \sigma)} & = \sum_{j=1}^{N(\delta)} \|g^j_\nu\|^2_{L^2(\Gamma, \sigma)} +\|q_\nu \|^2_{L^2(\Gamma, \sigma)}.
 \end{split}
 \end{equation}
 To see the orthogonality of parameters, we reorganize the decomposition. We say that $\gamma_\nu^j$ and $\gamma_\nu^k$ are orthogonal if for $j \neq k $, we have 
 \begin{equation}\label{orthogonality}
 \frac {r_\nu^j}{r_\nu^k} + \frac {r_\nu^k}{r_\nu^j} + \frac {| z_\nu^j-z_\nu^k|}{r_\nu^j} \to \infty, \text{ as } \nu\to\infty. 
 \end{equation}
Define 
$$h_\nu^1 =\sum_{j=1}^{N(\delta)} g_\nu^j -\sum_{\gamma_\nu^j\perp \gamma_\nu^1} g_\nu^j. $$
If there exists $2\le j_0\le N(\delta)$ such that $\gamma_\nu^{j_0}\perp \gamma_\nu^1$, then we define 
$$h_\nu^2 = \sum_{j=1}^{N(\delta)} g_\nu^j -\sum_{\gamma_\nu^j\perp \gamma_\nu^1 \atop \gamma_\nu^j\perp \gamma_\nu^{j_0}} g_\nu^j. $$
Repeating this argument a finite number of times, we rearrange the above sum. The $L^2$-norm orthogonality relation in \eqref{iteration2} holds, since the supports of the functions we consider are disjoint on the sphere. The $g_\nu^j$'s kept in the definition $h_\nu^1$ are such that $r_\nu^j$ are not orthogonal one to another. It is sufficiently to show that up to a subsequence, $\Phi_{\mathcal{C}_\nu^1} (g_\nu^j)$ is bounded. By construction, $\Phi_{\mathcal{C}_\nu^j}^{*} (g_\nu^j)$ are bounded by $c_\delta 1_{B(0,1)}(y)$. 

We compute, $\forall  y \in B(0,1))$ and $f$, a bounded function supported on the sphere, 
\begin{equation}\label{close-term}
\begin{split}
\Phi_{\mathcal{C}_\nu^1}^{*} \left(  \bigl( \Phi_{\mathcal{C}_\nu^j}^{*} \bigr)^{-1} f \right)(y) & = ( r_\nu^1)^{d/2} \left[ \bigl(\bigl( \Phi_{\mathcal{C}_\nu^j}^{*}\bigr)^{-1} f \bigr) \circ \Phi_{\mathcal{C}_\nu^1 }\right](y) \\
& = ( r_\nu^1)^{d/2} \bigl(\bigl( \Phi_{\mathcal{C}_\nu^j}^{*}\bigr)^{-1} f \bigr)\bigl( \Pi_{H_{z_\nu^1}}^{-1} \circ L_1^{-1} (r_\nu^1 y) \bigr)  \\
& =\bigl(  \frac {r_\nu^1}{r_\nu^j} \bigr)^{d/2} f\left( (r_\nu^1)^{-1} L_j \left( \Pi_{z_\nu^j}\bigl( \Pi_{H_{z_\nu^1}}^{-1} \circ L_1^{-1} (r_\nu^1 y) \bigr) \right) \right) . 
\end{split}
\end{equation}
This is the Euclidean coordinate representation of the function $f$ on the sphere $\mathbb{R}^d$. We remember that the last transform is through the cap $\mathcal{C}_\nu^j$. Now we show that it is bounded by $C_\delta1_{B(0,1)}(y). $
 If $$ \sup_\nu \left\{ \frac {r_\nu^1}{r_\nu^j} + \frac {r_\nu^j}{r_\nu^1} +\frac {|z_\nu^1-z_\nu^j|}{r_\nu^j} \right\}<\infty, $$
 then up to a subsequence, we may assume that 
$$ \frac {r_\nu^1}{r_\nu^j} \to c_1, $$
for some $0<c_1<\infty$. The same applies to $\frac {|z_\nu^1-z_\nu^j|}{r_\nu^j}$, that is to say, up to a subsequence, 
$$ \frac {|z_\nu^1-z_\nu^j|}{r_\nu^j} \to c_2$$
for some $c_2\ge 0$. Then
$$r_\nu^1\le Cr_\nu^j, |z_\nu^1-z_\nu^j|<Cr_\nu^j. $$
If $z\in \mathcal{C}(z_\nu^1, r_\nu^1)$, then 
$$|z-z_\nu^j|\le Cr_\nu^j. $$
Therefore the term in \eqref{close-term} is less than or equal to 
$$ C_\delta 1_{B(0,1)}(y). $$
The same applies to $\Phi_{\mathcal{C}_\nu^1}^{*} (g_\nu^k)$ for the terms $k$ such that $\gamma_\nu^k \not\perp \gamma_\nu^1. $ We also note that the total terms in the decomposition depends on $\delta$, where the constant $C_\delta$ comes from.  Therefore the proof for the Proposition \ref{prop-first-decomp} is complete. 
\end{proof}

\section{The second decomposition for each piece}\label{earthworm2}
For each component $f_\nu^j$ above, we decompose it further and establish the orthogonality of profiles. 
\subsection{Key propositions}\label{sec-key-prop}
In this section, we may assume that the function $f_\nu^j$ is supported on a cap $\mathcal{C}_\nu^j(z_\nu^j, r_\nu^j)$. The decomposition for $\widehat{f_\nu^j \sigma}$ is motivated by the rescaling relation,
\begin{equation}\label{eq-b5}
\begin{split}
&\left|\widehat{f_\nu^j \sigma}(x,t)\right|=\left|\int_{\mathcal{C}_\nu^j} e^{i(x,t)\cdot\xi}f_\nu^j(\xi)d\sigma(\xi)\right|\\
&=(r_\nu^j)^{d/2}\left|\int_{\mathbb{R}^d} e^{ix\cdot r_\nu^j y+it\sqrt{1-|r_\nu^j y|^2}}\frac {(r_\nu^j)^{d/2}f_\nu^j \circ \Pi^{-1}_{H_{z_\nu^j}} \circ L^{-1}_{z_\nu^j}(r_\nu^j y)}{ \sqrt{1-|r_\nu^j y|^2}}dy\right|\\
&=\left|(r_\nu^j)^{d/2} \int e^{i(r_\nu^j) x\cdot y-\frac {it(r_\nu^j)^2 |y|^2}{2}} e^{i(r_\nu^j)^2 t\bigl(\frac {\sqrt{1-(r_\nu^j)^2 |y|^2}-1}{(r_\nu^j)^2}+
\frac {|y|^2}{2}\bigr)}\frac {1}{(1-(r_\nu^j)^2 |y|^2)^{1/4}}\right. \times  \\
&\left.\qquad \qquad \qquad \qquad \qquad \times \dfrac {(r_\nu^j)^{d/2}f_\nu^j \circ \Pi^{-1}_{H_{z_\nu^j}} \circ L^{-1}_{z_\nu^j}(r_\nu^j y) }{(1-(r_\nu^j)^2 |y|^2)^{1/4}} dy\right|\\
&=\left|(r_\nu^j)^{d/2} e^{\frac {i(r_\nu^j)^2 t \Delta}{2}}\bigl(h_\nu((r_\nu^j)^2 t,\cdot) g^j_\nu(\cdot)\bigr)(r_\nu x)\right|,
\end{split}
\end{equation}where
\begin{equation}\label{eq-b250}
\begin{split}
h_\nu^j(t,y)&:=e^{it\bigl(\frac {\sqrt{1-(r_\nu^j)^2 |y|^2}-1}{(r_\nu^j)^2}+\frac {|y|^2}{2}\bigr)}\frac {1}{(1-(r_\nu^j)^2 |y|^2)^{1/4}},\\
g_\nu^j(y)&:= \dfrac {(r_\nu^j)^{d/2}f_\nu^j \circ \Pi^{-1}_{H_{z_\nu^j}} \circ L^{-1}_{z_\nu^j}(r_\nu^j y) }{(1-(r_\nu^j)^2 |y|^2)^{1/4}}.
\end{split}
\end{equation}
So a decomposition for $\widehat{f_\nu^j \sigma}$ immediately follows once we have a decomposition for $\{g_\nu^j\}$. The definitions in \eqref{eq-b250} is reminiscent of \cite[Equation (2.3)]{Frank-Lieb-Sabin:2007:maxi-sphere-2d}. In what follows in this section, we will suppress the superscripts $j$.

\begin{proposition}\label{prop-decomp}
Let $\{g_\nu\}$ be defined as above and let $C>0$ be the universal constant such that 
\begin{equation}\label{eq-b350}
|g_\nu| \le C1_{B(0,1)}. 
\end{equation}
Then there exists a sequence of functions $\{\phi^j\}_j\in L^2$ and $(x_\nu^j,t_\nu^j)\in \R^d\times \mathbb{R}$ and $e_\nu^l\in L^2$ such that
\begin{equation}\label{eq-b3}
g_\nu(y)=\sum_{j=1}^l e^{\frac {it^j_\nu |y|^2}{2}}e^{-ix^j_\nu y} \phi^j+e^l_\nu(y).
\end{equation} with the following properties: $\phi^j$ can be taken to be smooth functions supported on a ball of radius 1 uniformly in $\nu$. The parameters $\{(x_\nu^j,t_\nu^j)\}$ satisfy, for $k\neq j$,
\begin{equation}\label{eq-b4}
|x_\nu^k-x_\nu^j|+|t_\nu^k-t_\nu^j|\to \infty, \text{ as }\nu\to\infty.
\end{equation}
For each $l\ge 0$,
\begin{equation}\label{eq-L2-ortho}
\limsup_{\nu\to\infty}\left( \|g_\nu\|_{L^2(\Gamma, \sigma)}^2-\sum_{j=1}^l\|\phi^j\|_{L^2(\mathbb{R}^d)}^2-\|e_\nu^l\|^2_{L^2}\right)=0.
\end{equation}where $e_\nu^l$ satisfies
\begin{equation}\label{eq-b8}
\limsup_{l\to\infty}\limsup_{\nu\to\infty}\left\| r_\nu^{d/2} e^{\frac {itr_\nu^2\Delta}{2}}\bigl( h(r_\nu^2 t,\cdot)e_\nu^l\bigr)(r_\nu x)\right\|_{L^{2+4/d}(\mathbb{R}^{d+1})}=0.
\end{equation}where
$$e^{\frac {it\Delta}{2}}f(x)=\int e^{ixy-\frac {it|y|^2}{2}}f(y)dy.$$
\end{proposition}
The following result is for two profiles with the same center $z_\nu^j$ and $r_\nu^j$. More general decomposition is obtained in the next section when we combine Propositions \ref{prop-first-decomp} and \ref{prop-decomp}.
\begin{proposition}[Orthogonality I]\label{prop-ortho}
Let $\{h_\nu\}$ be defined as above and set \begin{equation}\label{eq-b2}
\begin{split}
G_\nu^k(y)& :=e^{\frac {it^k_\nu |y|^2}{2}}e^{-ix^k_\nu y} \phi^k,\\
G_\nu^j(y)& :=e^{\frac {it^j_\nu |y|^2}{2}}e^{-ix^j_\nu y} \phi^j.
\end{split}
\end{equation} where $\{(x_\nu^j,t_\nu^j)\}$ satisfying \eqref{eq-b4}. Then for $k\neq j$,
\begin{equation}\label{eq-b6}
\begin{split}
&\limsup_{\nu\to\infty} \left\| \Bigl(r_\nu^{d/2} e^{\frac {itr_\nu^2\Delta}{2}}\bigl( h_\nu(r_\nu^2 t,\cdot)G_\nu^k\bigr)(r_\nu x)\Bigr)\Bigl(r_\nu^{d/2} e^{\frac {itr_\nu^2\Delta}{2}}\bigl( h_\nu(r_\nu^2 t,y)G_\nu^j \bigr)(r_\nu x)\Bigr) \right\|_{L_{t,x}^{1+2/d}(\mathbb{R}^d)}=0, 
\end{split}
\end{equation}
where the integrand $ \bigl(r_\nu^{d/2} e^{\frac {itr_\nu^2\Delta}{2}}\bigl( h_\nu(r_\nu^2 t,\cdot)G_\nu^k\bigr)(r_\nu x)$ is understood as the Fourier transform of a surface measure on the sphere. 
\end{proposition}

\begin{proof}[Proof of Proposition \ref{prop-decomp}.]
We split the proof into two steps. We will follow the proof in \cite[Proposition 3.4]{Carles-Keraani:2007:profile-schrod-1d} and  present them for sake of completeness.

\textbf{Step 1.} For $(x_\nu,t_\nu)\in \R^{d+1}$, we define
$$T_\nu(g)(y)=e^{-\frac {it_\nu |y|^2}{2}} e^{ix_\nu y}g(y);$$
analogously $T_\nu^i$ for $(x_\nu^i,t_\nu^i)$ for $i\ge 1$, and $T_\nu^{-1}(g)(y)=e^{\frac {it_\nu |y|^2}{2}} e^{-ix_\nu y}g(y)$. Let $P^0$ denote the sequence $\{g_\nu\}_{\nu\ge 1}$. Then we define the set
$$\mathcal{W}(P^0)=\{w-\lim_{\nu\to\infty} T_\nu (P_\nu^0)(y) \text{ in } L^2:\, (x_\nu,t_\nu)\in \R^{d+1}\},$$
where $w-\lim g_\nu$ denotes a weak limit of $\{g_\nu\}$ in $L^2$. Define the blow-up criterion associated to $\mathcal{W}(P^0)$:
$$ \mu(P^0):=\sup\{\|\phi\|_{L^2}: \phi\in \mathcal{W}(P^0)\}.$$
Then for any $\phi\in \mathcal{W}(P^0)$, $\|\phi\|_{L^2} \le \limsup \|g_\nu\|_{L^2(\mathbb{R}^d)}=\limsup \|f_\nu\|_{L^2(\sigma)}$, which folllows from \cite[Theorem 2.11]{Lieb-Loss:2001} about the $L^2$-norm comparison for a weak convergent sequence.

If $\mu(P^0)=0$, then we set $l=0$, and $e_\nu^0=g_\nu $ for all $\nu\ge 1$. Otherwise, $\mu(P^0)>0$, then up to a subsequence, there exists nontrivial $\phi^1\in L^2$ and $(x_\nu^1,t_\nu^1)_{\nu\ge 1}$ such that
\begin{align}
\phi^1 &=w-\lim_{\nu\to\infty}T_\nu^1(P_\nu^0)(y), \\
\|\phi^1\|_{2}&\ge \frac 12 \mu(P^0).
\end{align}Let $P^1$ denote the sequence $\{g_\nu(y)-(T_\nu^1)^{-1}(\phi^1)(y)\}_{\nu\ge 1}$. It is not hard to see that
\begin{align}
&\label{eq-b23} w-\lim_{\nu\to\infty} T_\nu^1(P_\nu^1)=0,\\
&\label{eq-b24} \|g_\nu\|_{L^2(\sigma)}^2 -\|\phi^1\|_2^2-\|e_\nu^1\|^2_{L^2} \to 0, \text{ as }\nu\to\infty.
\end{align}Indeed, \eqref{eq-b23} follows from the definition. For \eqref{eq-b24}, 
$$\|e_\nu^1\|_{L^2}^2=\|T_v^1(g_\nu) -\phi^1\|_{L^2}^2= \|T_\nu^1(g_\nu)- \phi^1\|_{L^2}^2-2\operatorname{Re}\langle T_\nu^1(g_\nu) , \phi^1\rangle +\|\phi^1\|_{L^2}^2.$$
It implies that
$$ \|g_\nu\|_{L^2}^2-2\operatorname{Re}\langle T_\nu^1(g_\nu) , \phi^1\rangle+\|\phi^1\|_{L^2}^2-\|e_\nu^1\|_{L^2}^2 \to \|g_\nu\|_{L^2(\sigma)}^2 -\|\phi^1\|_2^2-\|e_\nu^1\|^2_{L^2} \to 0, \text{ as } \nu\to\infty. $$ 

For $P^1=\{g_\nu(y)-(T_\nu^1)^{-1}(\phi^1)(y)\}_{\nu\ge 1}$, we iteratively consider the set
$$\mathcal{W}(P^1)=\{w-\lim_{\nu\to\infty} T_\nu (P_\nu^1)\text{ in } L^2:\, (x_\nu, t_\nu)\in \R^{d+1}\}.$$
Then we test whether $\mu(P^1)>0$: if $\mu(P^1)=0$, then the algorithm stops; take $l=1$, and $e_\nu^1=P^1_\nu$; if not, then up to a subsequence, there exists nontrivial $\phi^2\in L^2$ and and $(x_\nu^2,t_\nu^2)_{\nu\ge 1}$ such that
\begin{align}
\phi^2 &=w-\lim_{\nu\to\infty}T_\nu^2(P_\nu^1)(y), \\
\|\phi^2\|_{2}&\ge \frac 12 \mu(P^1).
\end{align}
By similar considerations as in \eqref{eq-b23} and \eqref{eq-b24}, if setting $P_\nu^2= P_\nu^1-(T_\nu^2)^{-1}(\phi^2)$, then
\begin{align*}
 & w-\lim_{\nu\to\infty} T_\nu^2(P_\nu^2)=0,\\
 & \|g_\nu\|_{L^2(\sigma)}^2 -\sum_{j=1}^2 \|\phi^j\|_2^2- \|e_\nu^2\|_{L^2}^2 \to 0, \text{ as }\nu\to\infty.
\end{align*}

Now we claim that \eqref{eq-b4} holds. Otherwise, up to a subsequence we may assume that
$$|t_\nu^2-t_\nu^1|+|x_\nu^2-x_\nu^1|\to c, \text{ as }\nu\to\infty,$$for some $0<c<\infty$. In this case, the dominated convergence theorem gives, up to a subsequence,
$$T_\nu^2(T_\nu^1)^{-1} \text{ converges strongly in } L^2.$$
This will imply that \begin{equation}\label{eq-b28}
T_\nu^2(P_\nu^1)\to 0, \text{ weakly in } L^2,
 \end{equation} as $T_\nu^1(P_\nu^1) \to 0$ weakly in $L^2$  and the following relation holds, $$T_\nu^2(P_\nu^1)=T_\nu^2(T_\nu^1)^{-1} \bigl(T_\nu^1(P_\nu^1)\bigr). $$
But the claim in \eqref{eq-b28} is a contradiction to the existence of nontrivial $\phi^2$.  So \eqref{eq-b4} holds.

Iterating this process, a diagonalization argument produces a family of pairwise orthogonal
sequences $(x_\nu^j, t_\nu^j)$ and $\phi^j$ satisfying \eqref{eq-b3}, \eqref{eq-b4} and \eqref{eq-L2-ortho}. Since $\sum_j \|\phi^j\|_2^2\le \sup_\nu \|g_\nu\|_2^2 <\infty$ and $\mu(P^{l+1}) \le 2\|\phi^l\|_2$, we have
\begin{equation}\label{eq-b7}
\mu(P^{l})\to 0, \text{ as } l\to\infty.
\end{equation}
We give the following two remarks. (1). Since $g_\nu$ is supported on the unit ball and bounded, $\phi^j$ can be taken to be smooth functions supported on a ball of radius 1 uniformly in $\nu$. (2). The functions $\{e_\nu^l\}$ is supported in the unit ball and uniformly bounded in $l$ and $\nu$. Note that the proof of \eqref{eq-L2-ortho} yields a more precise result than \eqref{eq-L2-ortho}: for $l\ge 1$ and for $\psi\in L^2\cap L^\infty(\mathbb{R}^d)$,
\begin{equation}\label{eq-L2-ortho1}
\limsup_{\nu\to\infty}\left( \|\psi g_\nu\|_{L^2(\Gamma, \sigma)}^2-\sum_{j=1}^l\|\psi \phi^j\|_{L^2(\mathbb{R}^d)}^2-\|\psi e_\nu^l\|^2_{L^2}\right)=0.
\end{equation}
Then for the support of $\{e_\nu^l\}$, we consider $1_{E}$ for $E\subset \bigl( B(0,1)\bigr)^c$ and the measure of $E$, $|E| >0$, the Lebesgue measure of $E$. Then because $g_\nu$ and $\phi^j$ are supported in $B(0,1)$, 
$$\|1_E e_\nu^l\|_{L^2}=0, \text{ as } \nu\to \infty, $$
which implies $$ e_\nu^l =0,\text{ for sufficiently large } \nu; \text{thus } e_\nu^l \text{ is supported in } B(0,1) \text{ for sufficiently large }\nu. $$
Secondly for the boundedness of $\{e_\nu^l\}$: given any $A>0$, let $$E_A :=\{ x\in B(0,1):\, |e_\nu^l|>A\}. $$
Then we have 
$$ A^2 |E_A| \le \|1_{E_A} e_\nu^l \|_{L^2}^2 \le \|1_{E_A} g_\nu\|_{L^2}^2 \le \|g_\nu\|_{L^\infty}^2 |E_A|. $$
If $|E_A| \neq 0$, then $A\le \|g_\nu\|_{L^\infty}. $ This proves that $\|e_\nu^l\|_{L^\infty} \le \|g_\nu\|_{L^\infty}$ for sufficiently large $\nu$; for the definition of the $L^\infty$ norm, see e.g.,\cite[Page 184]{Folland:1999}. Then a diagonalization argument yields that the functions $\{e_\nu^l\}$ is supported in the unit ball and uniformly bounded in $l$ and $\nu$.

\textbf{Step 2.} At this step, we show that some localized restriction estimate $L^\infty(S^d)\to L^q_{t,x}$ for some $q<2+4/d$ and the information that $\lim \mu(P^l)=0$ will imply \eqref{eq-b8}. The restriction estimate we need is
\begin{mainthm}\label{le-local-restr}
There exists $q<2+4/d$. If $f\in L^\infty(S^d,\sigma)$, then
$$ \|\widehat{f\sigma}\|_{L^q_{t,x}} \le C\|f\|_{L^\infty(S^d, \sigma)}.$$
Similarly if $f\in L^\infty(B(0,R))$, then $$ \|e^{it\Delta}f\|_{L^q_{t,x}} \le C_R\|f\|_{L^\infty}.$$
\end{mainthm}
This follows from the adjoint linear restriction estimate for the sphere beyond Tomas-Stein's inequality in \cite{Tao:2003:paraboloid-restri}.

We continue proving \eqref{eq-b8}. There are still two cases to consider: $r_\nu \to 0$ or $r_\nu\to r_0 >0$. It suffices to consider the case $r_\nu\to 0$ as the other case is similar. Then by scaling, the norm on the left hand side of \eqref{eq-b8} is equivalent to $$\|e^{\frac {it\Delta}{2}}\bigl(h_\nu(t)e_\nu^l\bigr)\|_{L^{2+4/d}_{t,x}}.$$

From the way $e_\nu^l$ is obtained, we may assume that $e_\nu^l$ is compactly supported in the unit ball, and bounded uniformly in $\nu$ and $l$. Then by Lemma \ref{le-local-restr} and the real interpolation \cite[Theorem 1.3]{Stein-Weiss:1971:fourier-analysis}, \eqref{eq-b8} is reduced to
\begin{equation}\label{eq-b11}
\limsup_{l\to\infty}\limsup_{\nu\to\infty} \|e^{\frac {it\Delta}{2}}\bigl(h_\nu(t)e_\nu^l\bigr)\|_{L^\infty_{t,x}} =0.
\end{equation}
This follows from the fact that $\mu(P^l) \to 0$ as $l\to\infty.$ Indeed, there exists $(x^l_\nu,t^l_\nu)$ such that, up to a subsequence,
\begin{equation}\label{eq-b12}
\left|e^{\frac {it_\nu^l\Delta}{2}}\bigl(h_\nu(t_\nu^l)e_\nu^l\bigr)(x_\nu^l)\right|\sim\|e^{\frac {it\Delta}{2}}\bigl(h_\nu(t)e_\nu^l\bigr)\|_{L^\infty_{t,x}}.
\end{equation}
On the other hand, since $e_\nu^l$ is compactly supported, $$ e^{-\frac {it^l_\nu |y|^2}{2}}e^{ix^l_\nu y}h_\nu(t_\nu^l) e_\nu^l(y)=e^{-\frac {it^l_\nu |y|^2}{2}}e^{ix^l_\nu y}h_\nu(t_\nu^l) e_\nu^l(y)\phi(y)$$ for some suitable bump function $\phi$ adapted to the ball $B(0,1)$; taking integration in $y$ on both sides, we have
\begin{equation}\label{eq-b14}
e^{\frac {it_\nu^l\Delta}{2}}\bigl(h_\nu(t_\nu^l)e_\nu^l\bigr) (x_\nu^l)
=\langle e^{-\frac {it^l_\nu |y|^2}{2}}e^{ix^l_\nu y}h_\nu(t_\nu^l,y)e_\nu^l(y),\, \phi\rangle_{L^2}.
\end{equation}
The right hand side of the above is equal to 
$$ \int_{|y|\le 1} e^{ix_\nu^l y + it_\nu^l \frac {\sqrt{1-|r_\nu^l y|^2} -1}{(r_\nu^l)^2}} \frac {e_\nu^l (y) \phi(y)}{(1-|r_\nu^l y|^2)^{1/4}}dy. $$
By using the stationary phase estimate \cite[Page 334, Proposition 3]{Stein:1993}, since $0<r_\nu^l \le \frac 12$ for all $l,\nu$, if for a $l\ge 1$, $|t_\nu^l|\to \infty$ as $\nu\to \infty$, 
$$ \left|\int_{|y|\le 1} e^{ix_\nu^l y + it_\nu^l \frac {\sqrt{1-|r_\nu^l y|^2} -1}{(r_\nu^l)^2}} \frac {e_\nu^l (y) \phi(y)}{(1-|r_\nu^l y|^2)^{1/4}}dy\right| \le C |t_\nu^l|^{-d/2} \to 0, \text{ as } \nu\to \infty. $$ 
There is also the case where $\{t_\nu^l\}_\nu$ is bounded, so up to a subsequence, $\{t_\nu^l\}_\nu$ converges to a fixed number. Then  the sequence of functions $\{e_\nu^l\}_{\nu\ge 1}$ converges strongly in $L^2(\mathbb{R}^d)$ to a fixed function. So if $P^l:=\{e_\nu^l\}_{\nu\ge 1}$, by the definition of $\mu(P^l)$,
\begin{equation}
\text{ LHS }\eqref{eq-b12} \le \mu(P^l) \|\phi\|_{L^2} \to 0, \text{ as } \nu\to\infty,
\end{equation} since $\mu(P^l)\to 0$ as $l\to\infty.$ This finishes the proof of \eqref{eq-b8}. Therefore the proof of Proposition \ref{prop-decomp} is complete.
\end{proof}

Next we show that \eqref{eq-b4} implies the orthogonality result \eqref{eq-b6} in Proposition \ref{prop-ortho}.
\begin{proof}[Proof of \eqref{eq-b6} in Proposition \ref{prop-ortho}.] For given $j,k$, we know that $\phi^j$ and $\phi^k$ are smooth functions compactly supported in some ball of radius 1.

Recall that
$$ e^{\frac {it\Delta}{2}}\bigl( h_\nu(t,y)G_\nu^k\bigr)=\int e^{i(x-x_\nu^k)-\frac {i(t-t_\nu^k)|y|^2}{2}}e^{i(t-t_\nu^k)\bigl(\frac {\sqrt{1-r^2_\nu |y|^2}-1}{r_\nu^2}+\frac {|y|^2}{2}\bigr)} \frac {\phi^k(y)}{(1-r_\nu^2 |y|^2)^{1/4}} dy. $$
Likewise for $e^{\frac {it\Delta}{2}}\bigl( h_\nu(t,y)G_\nu^j\bigr)$. By a change of variables, we need to show
\begin{equation}\label{eq-b15}
\begin{split}
\left\| e^{i\frac {t-(t_\nu^j-t_\nu^k)}{2}\Delta}\Bigl( e^{i\bigl(t-(t_\nu^j-t_\nu^k)\bigr)\bigl(
\frac {\sqrt{1-r_\nu^2 |y|^2}-1}{r_\nu^2}+\frac {|y|^2}{2}\bigr)} \frac {\phi^j(y)}{ (1-r_\nu^2 |y|^2)^{1/4}}  \Bigr)\bigl(x-(x_\nu^j-x_\nu^k)\bigr) \right. \\
\quad \quad \times \left. e^{i\frac {t\Delta}{2}}\Bigl( e^{it\bigl(
\frac {\sqrt{1-r_\nu^2 |y|^2}-1}{r_\nu^2}+\frac {|y|^2}{2}\bigr)} \frac {\phi^k(y)}{(1-r_\nu^2 |y|^2)^{1/4}}  \Bigr)\right\|_{L^{1+2/d}_{t,x}}
\end{split}
\end{equation} goes to zero as $\nu$ goes to infinity.

For a large $N\gg 1$, set
$$\Omega_N:=\{(t,x): |t|+|x|\le N\},\quad \Omega_{N,\nu}:=\Omega_N-(t_\nu^j-t_\nu^k, x_\nu^j-x_\nu^k).$$
We first claim that, for $\Omega=\Omega_N$ or $\Omega_{N,\nu}$,
\begin{equation}\label{eq-b16}
\begin{split}
\int_{\Omega_N^c}\left| e^{i\frac {t-(t_\nu^j-t_\nu^k)}{2}\Delta}\Bigl( e^{i\bigl(t-(t_\nu^j-t_\nu^k)\bigr)\bigl(
\frac {\sqrt{1-r_\nu^2 |y|^2}-1}{r_\nu^2}+\frac {|y|^2}{2}\bigr)} \frac {\phi^j(y)}{ (1-r_\nu^2 |y|^2)^{1/4}}  \Bigr)\bigl(x-(x_\nu^j-x_\nu^k)\bigr) \right. \\
\quad \quad \times \left. e^{i\frac {t\Delta}{2}}\Bigl( e^{it\bigl(
\frac {\sqrt{1-r_\nu^2 |y|^2}-1}{r_\nu^2}+\frac {|y|^2}{2}\bigr)} \frac {\phi^k(y)}{(1-r_\nu^2 |y|^2)^{1/4}}  \Bigr)\right|^{1+\frac 2d} dtdx
\end{split}
\end{equation} goes to zero as $N$ goes to infinity uniformly in $\nu$. Here $\Omega^c:=\R^{d+1}\setminus \Omega.$

We write
\begin{equation}\label{eq-b17}
\begin{split}
& e^{i\frac {t\Delta}{2}}\Bigl( e^{it\bigl(\frac {\sqrt{1-r_\nu^2 |y|^2}-1}{r_\nu^2}+\frac {|y|^2}{2}\bigr)} \frac {\phi^j(y)}{(1-r_\nu^2 |y|^2)^{1/4}} \Bigr)(x) \\
&\quad =\int e^{ixy+it \frac {\sqrt{1-|r_\nu y|^2}-1}{r_\nu^2}}  \frac {\phi^j(y)}{(1-r_\nu^2 |y|^2)^{1/4}} dy.
\end{split}
\end{equation}
We observe that, for $0<r_\nu\le \frac 12$ and all $|y|\le 1$, $\bigl| \partial^2_{y_i }  \frac {\sqrt{1-|r_\nu y|^2}-1}{r_\nu^2}  \bigr|\ge c>0$ for some $c>0$ uniformly in $\nu$. Then by the stationary phase estimate \cite[p.334]{Stein:1993}, the quantity in \eqref{eq-b17} is bounded by  $C|t|^{-d/2}$. 
 On the other hand, by integration by parts, if $|x|>10|t|$, for all sufficiently large $\nu$,
\begin{equation}\label{eq-b18}
\left|e^{i\frac {t\Delta}{2}}\Bigl( e^{it\bigl(\frac {\sqrt{1-r_\nu^2 |y|^2}-1}{r_\nu^2}+\frac {|y|^2}{2}\bigr)}\frac {\phi^j(y)}{(1-r_\nu^2 |y|^2)^{1/4}}  \Bigr)(x)\right|\le C_N|x|^{-N}, \forall N \in \mathbb{N}.
\end{equation} Furthermore there always hold a trivial bound, for all $x,t$,
\begin{equation}\label{eq-b19}
\left|e^{i\frac {t\Delta}{2}}\Bigl( e^{it\bigl(\frac {\sqrt{1-r_\nu^2 |y|^2}-1}{r_\nu^2}+\frac {|y|^2}{2}\bigr)}\frac {\phi^j(y)}{(1-r_\nu^2 |y|^2)^{1/4}}  \Bigr)(x)\right|\le C.
\end{equation}
Here $C$ depends on the function $\phi^j$. 

We are now ready to prove \eqref{eq-b16} when $\Omega=\Omega_N$; the case where $\Omega=\Omega_{N,\nu}$ is similar and so will be omitted. By Cauchy-Schwarz,
\begin{equation}
\begin{split}
\text{ LHS of }\eqref{eq-b16} &\le C \left\| e^{i\frac {t-(t_\nu^j-t_\nu^k)\Delta}{2}}\Bigl( e^{i\bigl(t-(t_\nu^j-t_\nu^k)\bigr)\bigl(
\frac {\sqrt{1-r_\nu^2 |y|^2}-1}{r_\nu^2}+\frac {|y|^2}{2}\bigr)}\frac {\phi^j(y)}{(1-r_\nu^2 |y|^2)^{1/4}}  \Bigr) \bigl(x-(x_\nu^j-x_\nu^k)\bigr)\right\|^{1+2/d}_{L^{2+4/d}} \\
&\qquad \times \left\|e^{i\frac {t\Delta}{2}}\Bigl( e^{it\bigl(
\frac {\sqrt{1-r_\nu^2 |y|^2}-1}{r_\nu^2}+\frac {|y|^2}{2}\bigr)}\frac {\phi^k(y)}{(1-r_\nu^2 |y|^2)^{1/4}} \Bigr)\right\|^{1+2/d}_{L^{2+4/d}_{t,x}(\Omega^c_N)}\\
&\le C \left\| e^{i\frac {t}{2}\Delta}\Bigl( e^{it\bigl(\frac {\sqrt{1-r_\nu^2 |y|^2}-1}{r_\nu^2}+\frac {|y|^2}{2}\bigr)}\frac {\phi^j(y)}{(1-r_\nu^2 |y|^2)^{1/4}} \Bigr)\right\|^{1+2/d}_{L^{2+4/d}}\times \\
&\qquad \times \left\|e^{i\frac {t\Delta}{2}}\Bigl( e^{it\bigl(
\frac {\sqrt{1-r_\nu^2 |y|^2}-1}{r_\nu^2}+\frac {|y|^2}{2}\bigr)}\frac {\phi^k(y)}{(1-r_\nu^2 |y|^2)^{1/4}}  \Bigr)\right\|^{1+2/d}_{L^{2+4/d}_{t,x}(\Omega^c_N)}\\
\end{split}
\end{equation}
On the one hand, by Tomas-Stein's inequality for the sphere and a change of variables  the first term above equals
\begin{equation}\label{eq-b20}
\begin{split}
&\left\|e^{i\frac {t}{2}\Delta}\Bigl( e^{it\bigl(\frac {\sqrt{1-r_\nu^2 |y|^2}-1}{r_\nu^2}+\frac {|y|^2}{2}\bigr)}\frac {\phi^j(y)}{(1-r_\nu^2 |y|^2)^{1/4}}  \Bigr)\right\|_{L^{2+4/d}_{t,x}}\\
&=\left\|\int e^{ixy+it\frac {\sqrt{1-r_\nu^2 |y|^2}}{r_\nu^2}}\frac {\phi^j(y)}{ (1-r_\nu^2 |y|^2)^{1/4}}  dy \right\|_{L^{2+4/d}_{t,x}}\\
&=\left\|\int_{|y|\le 1/2} e^{ixy+it\sqrt{1-|y|^2}} \bigl(r_\nu^{-d/2}\phi^j(r_\nu^{-1}y)( 1-|y|^2)^{1/4} \frac {dy}{\sqrt{1-|y|^2}} \right\|_{L^{2+4/d}_{t,x}}\\
&=\left\|\int e^{i(x,t)\xi} \bigl(r_\nu^{-d/2}\phi^j(r_\nu^{-1}y)  \sqrt{1-|y|^2}\bigr) d\sigma \right\|_{L^{2+4/d}_{t,x}}\\
&\le C \left( \int |r_\nu^{-d/2}\phi^j(r_\nu^{-1}y)|^2 \sqrt{1-|y|^2} d\sigma(y) \right)^{1/2}\\
&\le C \left( \int |r_\nu^{-d/2}\phi^j(r_\nu^{-1}y)|^2 dy \right)^{1/2}=C\|\phi^j\|_{L^2}
\end{split}
\end{equation} for some $C>0$ independent of $\nu$.  Then we have estimates \eqref{eq-b17}, \eqref{eq-b18} and \eqref{eq-b19}, which yields that
\begin{equation}\label{eq-b21}
\left\|e^{i\frac {t\Delta}{2}}\Bigl( e^{it\bigl(
\frac {\sqrt{1-r_\nu^2 |y|^2}-1}{r_\nu^2}+\frac {|y|^2}{2}\bigr)}\frac {\phi^k(y)}{(1-r_\nu^2 |y|^2)^{1/4}} \Bigr)\right\|^{1+2/d}_{L^{2+4/d}_{t,x}(\Omega^c_N)}\to 0, \text{ uniform in } \nu \text{ as } N\to \infty.
\end{equation}
Therefore we have established \eqref{eq-b16}, which can be also established by using the Tomas-Stein inequality for the sphere and the dominated convergence theorem. To finish the proof of \eqref{eq-b15}, we need to show that, for a fixed $N\gg 1$,
\begin{equation}\label{eq-b22}
\begin{split}
\int_{\Omega_N\cap\Omega_{N,\nu}} \left| e^{i\frac {t-(t_\nu^j-t_\nu^k)}{2}\Delta}\Bigl( e^{i\bigl(t-(t_\nu^j-t_\nu^k)\bigr)\bigl(
\frac {\sqrt{1-r_\nu^2 |y|^2}-1}{r_\nu^2}+\frac {|y|^2}{2}\bigr)} \frac {\phi^j(y)}{ (1-r_\nu^2 |y|^2)^{1/4}}  \Bigr)\bigl(x-(x_\nu^j-x_\nu^k)\bigr) \right. \\
\quad \quad \times \left. e^{i\frac {t\Delta}{2}}\Bigl( e^{it\bigl(
\frac {\sqrt{1-r_\nu^2 |y|^2}-1}{r_\nu^2}+\frac {|y|^2}{2}\bigr)} \frac {\phi^k(y)}{ (1-r_\nu^2 |y|^2)^{1/4}}  \Bigr)\right|^{1+\frac 2d} dtdx
\end{split}
\end{equation} goes to zero as $\nu$ goes to infinity. It actually holds as
$$\operatorname{measure}(\Omega_N\cap \Omega_{N,\nu}) \to 0, \text{ as } \lim_{\nu\to\infty} |t_\nu^j-t_\nu^k|+|x_\nu^j-x_\nu^k|=\infty,$$
and we can apply $L^\infty_{t,x}$-bounds to both integrals. Therefore the proof of \eqref{eq-b6} is complete.
\end{proof}

\section{The full decomposition}\label{earthworm3}
The orthogonality result of \eqref{eq-b6} in Proposition \ref{prop-ortho} is established under the assumption that the sequence $\{f_\nu^j\}$ is all supported on the same cap $\mathcal{C}(z_\nu^j,r_\nu^j)$. In this section, we combine Propositions \ref{prop-first-decomp}, \ref{prop-decomp} to obtain a full decomposition for $\widehat{f_\nu\sigma}$. We need a lemma on the adjoint bilinear restriction estimate for paraboloids from Tao \cite{Tao:2003:paraboloid-restri} and Bourgain \cite{Bourgain:1998:refined-Strichartz-NLS}.
\begin{lemma}\label{le-bilinear-sphere}
Assume $f_1$ and $f_2$ are supported on the two caps $\mathcal{C}_1(z_1, r)$ and $\mathcal{C}_2(z_2, r)$ on the sphere $S^d$. Let $N:=\frac {|z_1-z_2|}{r}\gg 1$.  Then for any $q>\frac {d+3}{d+1}$, there exists $\alpha=\alpha(d)>0$, 
\begin{equation}\label{eq-tao-bilinear}
\|\widehat{f_1\sigma} \widehat{f_2\sigma}\|_{L^q_{t,x}} \lesssim \left( \frac {|z_1-z_2|}{r}\right)^{-\alpha}  \|f_1\|_{L^2} \|f_2\|_{L^2}. . 
\end{equation}
\end{lemma}
\begin{proof}
We recall the definition of the rescaled-map in Definition \ref{rescaled-maps}. We express the sphere surface measure in $\widehat{f\sigma}$ in terms of the plane coordinates, we see that
\begin{align*}
\left| \widehat{f_1\sigma} (t,x) \right| &= \left|  r^{d/2} \int e^{ix\cdot ry+ ir^2t\frac {\sqrt{1-|ry|^2}-1}{r^2}}  \dfrac {r^{d/2}f_1\circ \Pi^{-1}_{H_{z}} \circ L^{-1}_{z}(r y) }{(1-|ry|^2)^{1/4}} dy \right|, \\
\left| \widehat{f_2\sigma} (t,x)\right| &= \left| r^{d/2} \int e^{ix\cdot ry+ ir^2t\frac {\sqrt{1-|ry|^2}-1}{r^2}}  \dfrac {r^{d/2}f_2 \circ \Pi^{-1}_{H_{z}} \circ L^{-1}_{z}(r y) }{(1-|ry|^2)^{1/4}} dy \right|. 
\end{align*} 
Then by the change of variables, we see that 
\begin{align*}
\|\widehat{f_1\sigma} \widehat{f_2\sigma}\|_{L^{\frac {d+2}{d}}_{t,x} (\mathbb{R}\times \mathbb{R}^d)}& 
 =  \|\int e^{ix\cdot y + it \frac {\sqrt{1-|ry|^2}-1}{r^2}}  \dfrac {r^{d/2}f_1\circ \Pi^{-1}_{H_{z}} \circ L^{-1}_{z}(r y) }{(1-|ry|^2)^{1/2}} dy \\
& \quad\quad\times  \int e^{ix\cdot y + it \frac {\sqrt{1-|ry|^2}-1}{r^2}}  \dfrac {r^{d/2}f_2\circ \Pi^{-1}_{H_{z}} \circ L^{-1}_{z}(r y) }{ (1-|ry|^2)^{1/2}}  dy\|_{L^{\frac {d+2}{d}}_{t,x}} .
\end{align*}
By the triangle inequality $|ry_1-r y_2| \sim |z_1 -z_2|+2r$ and $y_i$ are points in balls of radius $1$ for $i=1,2$ with one ball centered at the origin,  $|y_1-y_2| \sim \frac {|z_1-z_2|}{r}+2$, as indicated in Figure \ref{fig-2}. Thus we see the supports of the two functions in the integrands are separated by $N\sim \frac {|z_1-z_2|}{r}$ since $\frac {|z_1-z_2|}{r}\gg 1$. By interpolation, the $L_{t,x}^{\frac {d+2}{d}}$ is controlled by the $L^{q_1}_{t,x}$ with some $\frac {d+3}{d+1}<q_1<\frac {d+2}{d}$ and $L^2_{t,x}$. That is to say, there exists $\theta\in (0,1)$ such that 
$$ \frac {d}{d+2} = \frac {1-\theta}{q_1}+ \frac {\theta}{2}. $$  

\begin{figure}[h]
        \centering
        \includegraphics[width=0.8\textwidth]{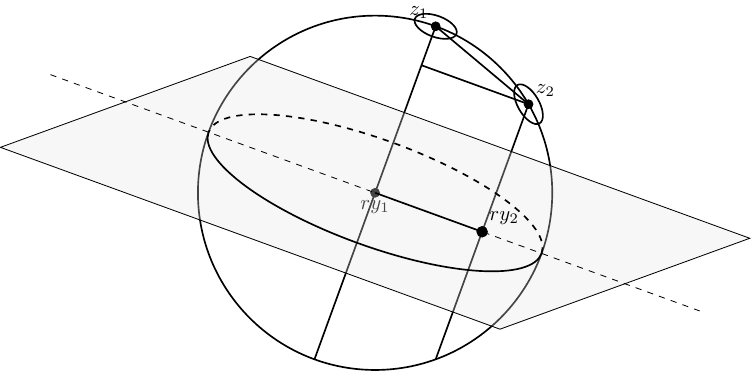}
        \caption{Distant caps interact weakly.}
        \label{fig-2}
    \end{figure}
    
Then by Tao's bilinear restriction estimate for elliptic phases  \cite[The third remark in Section 9]{Tao:2003:paraboloid-restri}, the $L^{q_1}_{t,x}$ norm
\begin{align*}
& \|\int e^{ix\cdot y + it \frac {\sqrt{1-|ry|^2}-1}{r^2}}  \dfrac {r^{d/2}f_1\circ \Pi^{-1}_{H_{z}} \circ L^{-1}_{z}(r y) }{(1-|ry|^2)^{1/2}} dy  \int e^{ix\cdot y + it \frac {\sqrt{1-|ry|^2}-1}{r^2}}  \dfrac {r^{d/2}f_2\circ \Pi^{-1}_{H_{z}} \circ L^{-1}_{z}(r y) }{ (1-|ry|^2)^{1/2}}  dy\|_{L^{q_1}_{t,x}}  \\
 &\lesssim  N^{d-\frac {d+2}{q_1}} \|\dfrac {r^{d/2}f_1\circ \Pi^{-1}_{H_{z}} \circ L^{-1}_{z}(r y) }{(1-|ry|^2)^{1/2}}  \|_{L^2(\mathbb{R}^d)} \|\dfrac {r^{d/2}f_2\circ \Pi^{-1}_{H_{z}} \circ L^{-1}_{z}(r y) }{(1-|ry|^2)^{1/2}}  \|_{L^2(\mathbb{R}^d)} \\
 &\lesssim N^{d-\frac {d+2}{q_1}} \|f_1\|_{L^2(S^d,\sigma)} \|\|_{L^2(S^d,\sigma)}. 
\end{align*}

By the Bourgain bilinear restriction estimate \cite[Lemma 111]{Bourgain:1998:refined-Strichartz-NLS} and \cite[Theorem 4.18]{Killip-Visan:2008:clay-lecture-notes}, the $L^2_{t,x}$ is bounded by
\begin{align*}
& \|\int e^{ix\cdot y + it \frac {\sqrt{1-|ry|^2}-1}{r^2}}  \dfrac {r^{d/2}f_1\circ \Pi^{-1}_{H_{z}} \circ L^{-1}_{z}(r y) }{(1-|ry|^2)^{1/2}} dy  \int e^{ix\cdot y + it \frac {\sqrt{1-|ry|^2}-1}{r^2}}  \dfrac {r^{d/2}f_2\circ \Pi^{-1}_{H_{z}} \circ L^{-1}_{z}(r y) }{ (1-|ry|^2)^{1/2}}  dy\|_{L^2_{t,x}}  \\
 &\lesssim  N^{-\frac 12 } \|\dfrac {r^{d/2}f_1\circ \Pi^{-1}_{H_{z}} \circ L^{-1}_{z}(r y) }{(1-|ry|^2)^{1/2}}  \|_{L^2(\mathbb{R}^d)} \|\dfrac {r^{d/2}f_2\circ \Pi^{-1}_{H_{z}} \circ L^{-1}_{z}(r y) }{(1-|ry|^2)^{1/2}}  \|_{L^2(\mathbb{R}^d)} \\
 &\lesssim N^{-\frac 12} \|f_1\|_{L^2(S^d,\sigma)} \|\|_{L^2(S^d,\sigma)}. 
\end{align*}
Therefore
\begin{align*}
&\|\widehat{f_1\sigma} \widehat{f_2\sigma}\|_{L^{\frac {d+2}{d}}_{t,x} (\mathbb{R}\times \mathbb{R}^d)} \lesssim 
\left(N^{d-\frac {d+2}{q_1}} \|f_1\|_{L^2(S^d,\sigma)} \|\|_{L^2(S^d,\sigma)} \right)^{1-\theta} \left(N^{d-\frac {d+2}{q_1}} \|f_1\|_{L^2(S^d,\sigma)} \|\|_{L^2(S^d,\sigma)} \right)^\theta  \\
&\lesssim N^{\bigl(d-\frac {d+2}{q_1}\bigr)(1-\theta)} N^{-\frac {\theta}{2}}\|f_1\|_{L^2(S^d,\sigma)} \|\|_{L^2(S^d,\sigma)}. 
\end{align*}
We set $\alpha =\bigl(\frac {d+2}{q_1}-d\bigr)(1-\theta) +\frac {\theta}{2}$. Therefore we have proved Proposition \ref{le-bilinear-sphere}. 
\end{proof}
\begin{remark}\label{re-connectCS}
As coined in \cite{Christ-Shao:extremal-for-sphere-restriction-I-existence}, this is known as ``distant caps interact weakly". When $d=2$, $\frac {d+2}{d}=2$, there is no need to choose $q<2$. In this case, the proof above uses Bourgain's bilinear restriction estimate in \cite[Lemma 111]{Bourgain:1998:refined-Strichartz-NLS} and \cite[Theorem 4.18]{Killip-Visan:2008:clay-lecture-notes}. This provides an alternative way to the estimate in \cite[Lemma 7.5]{Christ-Shao:extremal-for-sphere-restriction-I-existence}, where Christ and the first author use the structure of convolution of the sphere surface measures. 
\end{remark}

\begin{proposition}[Orthogonality II]\label{prop-full-decomp}
Let $$\limsup_{\nu\to\infty} \frac {r_\nu^j}{r_\nu^k}+\frac {r_\nu^k}{r_\nu^j}+\frac {|z_\nu^j-z_\nu^k|}{r_\nu^j}=\infty. $$ We have that, for any $N\ge 1$, as $\nu\to\infty$,
\begin{equation}\label{eq-b38}
\lim_{\nu\to\infty} \left\| \Bigl((r_\nu^j)^{d/2} e^{\frac {it(r_\nu^j)^2\Delta}{2}}\bigl( h_\nu((r_\nu^j)^2 t,\cdot)G_\nu^j\bigr)(r_\nu^j x)\Bigr)\Bigl((r_\nu^k)^{d/2} e^{\frac {it(r_\nu^k)^2\Delta}{2}}\bigl( h_\nu((r_\nu^k)^2 t,\cdot)G_\nu^k \bigr)(r_\nu^k x)\Bigr) \right\|_{L^{1+2/d}}=0
\end{equation} for $G_\nu^j$, $G_\nu^k$ defined in \eqref{eq-b2}.  
\end{proposition}

\begin{proof}
We assume that, up to subsequences, $\frac{r_\nu^j}{r_\nu^k}\to \infty$ as $\nu\to\infty$. It is further split into two cases.
\begin{itemize}
\item Case 1, $(r_\nu^j,r_\nu^k) \to (0,0)$ as $\nu\to \infty$. We change variables as follows, $r_\nu^k x $ changes to $x$, and $(r_\nu^k)^2 t$ changes to $t$, where we have abused the notations. Then the norm on the left hand side of \eqref{eq-b38} equals
    \begin{equation}\label{eq-b39}
    \begin{split}
    &\left(\frac {r_\nu^j}{r_\nu^k}\right)^{d/2}\left\| e^{it \bigl(\frac {r_\nu^j}{r_\nu^k}\bigr)^2\Delta/2}\bigl( h_\nu(\bigl(\frac {r_\nu^j}{r_\nu^k}\bigr)^2t,y)G_\nu^j\bigr)(\frac {r_\nu^j}{r_\nu^k}x)\Bigr)\,e^{\frac {it\Delta}{2}}\bigl( h_\nu(t,y)G_\nu^k \bigr)(x)\right\|_{L^{1+2/d}}\\
    \quad &=\left(\frac {r_\nu^j}{r_\nu^k}\right)^{d/2}\left\| \int e^{i(\frac {r_\nu^j}{r_\nu^k}x-x_\nu^j)y+i\bigl((\frac {r_\nu^j}{r_\nu^k})^2t-t_\nu^j\bigr)\frac {\sqrt{1-|r_\nu^j y|^2}-1}{(r_\nu^j)^2}} \dfrac  {\phi^j(y)}{( 1-|r_\nu^j y|^2)^{1/4} } dy \right.\\
  &\quad\quad  \times \int \left.e^{i(x-x_\nu^k)y+i\bigl(t-t_\nu^k\bigr)\frac {\sqrt{1-|r_\nu^k y|^2}-1}{(r_\nu^k)^2}}\dfrac {\phi^k (y)}  {(1-|r_\nu^k y|^2)^{1/4} }dy\right\|_{L^{1+2/d}}.
    \end{split}
    \end{equation}
    If $(r_\nu^j,r_\nu^k)\to (0,0)$, then $\frac {\sqrt{1-|r_\nu^j y|^2}-1}{(r_\nu^j)^2} \to -\frac {|y|^2}{2}$ as $\nu\to\infty$; furthermore, we know that $\phi^j$ and $\phi^k$ can be taken to be smooth functions with compact supports in a ball of radius $1$ and uniform $L^\infty$ bound. These can be used via the Strichartz estimate for the quadratic surfaces \cite{Strichartz-1977} to show
    \begin{equation}\label{eq-b40}
    \left\|\int e^{ixy+it\frac {\sqrt{1-|r_\nu^j y|^2}-1}{(r_\nu^j)^2}} \dfrac  {\phi^j(y)}{ ( 1-|y|^2)^{1/4} } dy\right\|_{L^{2+4/d}(|x|+|t|>R)} \to 0,
    \end{equation} as $R\to \infty$, uniformly in $\nu$. Indeed, by the change of variables, we have 
    $$( r_\nu^j)^{-d} \| \int e^{i\frac {x}{r_\nu^j} y+i\frac {t}{(r_\nu^j)^2} \sqrt{1-|y|^2}}   \dfrac  {\phi^{j} (\frac {y}{r_\nu^j}) }{( 1-|\frac {y}{r_\nu^j}|^2)^{1/4} } dy \|_{L^{2+\frac d4}_{t,x} \big(|x|+|t| \ge R \bigr) }. $$
    After changing variables $(r_\nu^j)^{-1}x \to x$ and $(r_\nu^j)^{-2} t \to t$, we see that 
    \begin{equation*}
    \begin{split}
    & \| \int e^{ix\cdot y+i t  \sqrt{1-|y|^2}}   \dfrac  { r_\nu^{-\frac d2}  \phi^{j} (\frac {y}{r_\nu^j}) }{( 1-|\frac {y}{r_\nu^j}|^2)^{1/4} } \sqrt{1-|y|^2}\frac {dy}{\sqrt{1-|y||^2}} \|_{L^{2+\frac d4}_{t,x} \big( r_\nu^j |x|+(r_\nu^j)^2 |t| \ge R \bigr) }\\
     &=\| \int e^{ix\cdot y+i t  \sqrt{1-|y|^2}}    \dfrac  { r_\nu^{-\frac d2}  \phi^{j} (\frac {y}{r_\nu^j}) }{( 1-|\frac {y}{r_\nu^j}|^2)^{1/4} }  \sqrt{1-|y|^2} d\sigma(y) \|_{L^{2+\frac d4}_{t,x} \big( r_\nu^j |x|+(r_\nu^j)^2 |t| \ge R \bigr) }. 
    \end{split}
    \end{equation*}
    The $L^{2+\frac 4d}_{t,x}$ norm is bounded by $\|   \dfrac  { r_\nu^{-\frac d2}  \phi^{j} (\frac {y}{r_\nu^j}) }{( 1-|\frac {y}{r_\nu^j}|^2)^{1/4} }  \sqrt{1-|y|^2}\|_{L^2(S^d, \sigma)}$ by using the Strichartz estimate. If we expand this norm with the use of expressing the sphere surface measure $d\sigma$ in terms of the $y$-coordinates, and change variables, we see that is bounded by  $\|\phi^j\|_{L^2}$. 
    We observe that $$\{r_\nu^j |x|+(r_\nu^j)^2 |t| \ge R \}\subset \{|x|+|t| \ge R/r_\nu^j\} \subset \{|x|+|t| \ge R\}. $$
    since $r_\nu^j \in (0,\frac 12)$. By the Strichartz estimate and the dominated convergence theorem, we see that it converges to zero as $\nu$ goes to infinity. Likewise for the $k$-th term. Hence we are led to consider the integration over the region $\Omega_\nu^j\cap \Omega_\nu^k$, where
    \begin{equation}\label{eq-b41}
    \begin{split}
    \Omega_\nu^j &:= \{\left| \frac {r_\nu^j}{r_\nu^k} x-x_\nu^j\right|+\left|(\frac {r_\nu^j}{r_\nu^k})^2t-t_\nu^j\right|\le R\},\\
     \Omega_\nu^k &:= \{\left|x-x_\nu^k\right|+\left|t-t_\nu^k\right|\le R\}
    \end{split}
    \end{equation}for a large $R>0$. Then continuing \eqref{eq-b39}, we see that it is bounded by
    \begin{equation}\label{eq-b42}
    C_R \|\phi^j\|_1 \|\phi^k\|_1 \left(\frac {r_\nu^j}{r_\nu^k}\right)^{d/2}\min \{ \left(\frac {r_\nu^j}{r_\nu^k}\right)^{-d},\,1\}\le C_R \|\phi^j\|_1 \|\phi^k\|_1  \min\{ \left(\frac {r_\nu^j}{r_\nu^k}\right)^{d/2},\,\left(\frac {r_\nu^j}{r_\nu^k}\right)^{-d/2}\}\to 0,
    \end{equation} as $\nu\to \infty$.

\item Case 2, $(r_\nu^j,r_\nu^k) \to (c,0)$ for some $c>0$. It is similar to Case 1.
\end{itemize}

Then we may assume that $r_\nu^j\equiv r_\nu^k$ for all $\nu$. In this case, up to subsequences, we consider $ \frac {|z_\nu^j-z_\nu^k|}{r_\nu^j}\to \infty$ as $\nu\to \infty$. This condition easily implies $r_\nu^j\to 0$ as $\nu\to \infty$.

As in Case 1 above, we estimate
\begin{equation}\label{eq-b43}
\begin{split}
&\left\| \int e^{i(x-x_\nu^j)y+i((t-t_\nu^j)\frac {\sqrt{1-|r_\nu^j y|^2}-1}{(r_\nu^j)^2}}\phi^j (y)dy\,\int e^{i(x-x_\nu^k)y+i\bigl(t-t_\nu^k\bigr)\frac {\sqrt{1-|r_\nu^k y|^2}-1}{(r_\nu^k)^2}}\phi^k (y)dy\right\|_{L^{1+2/d}} \\
&\lesssim (r_\nu^j)^{-d}\left\| \int e^{i\frac {x-x_\nu^j}{r_\nu^j } y+i \frac { t-t_\nu^j}{( r_\nu^j)^2} \sqrt{1-|y|^2}} ( r_\nu^j)^{-\frac d2}\phi^j (\frac {y}{r_\nu^j})dy\right. \\
&\quad \left. \times \int e^{i\frac {x-x_\nu^k}{r_\nu^j } y+i \frac { t-t_\nu^k}{( r_\nu^j)^2} \sqrt{1-|y|^2} } ( r_\nu^j)^{-\frac d2}\phi^k (\frac {y}{r_\nu^j}) \right\|_{L^{1+2/d}}\\
&\lesssim \left\| \int e^{i( x-(r_\nu^j)^{-1}x_\nu^j )\cdot  y+i( t-(r_\nu^j)^{-2} t_\nu^j ) \sqrt{1-| y|^2} }  ( r_\nu^j)^{-\frac d2}\phi^j (\frac {y}{r_\nu^j})dy\right.  \\
&\quad \times \left. \int e^{i( x-(r_\nu^j)^{-1}x_\nu^k )\cdot  y+i( t-(r_\nu^j)^{-2} t_\nu^k ) \sqrt{1-| y|^2}}  ( r_\nu^j)^{-\frac d2}\phi^j (\frac {y}{r_\nu^j})dy \right\|_{L^{1+2/d}}.
\end{split}
\end{equation}  We recognize it is an adjoint bilinear restriction estimate for sphere, where the the lift-function-to-sphere $\phi^j(\frac {\cdot}{r_\nu^j})$ and $\phi^k(\frac {\cdot}{r_\nu^j})$ are supported in cap $\mathcal{C}_\nu^j$ and $\mathcal{C}_\nu^k$, respectively.  Moreover we observe that 
\begin{equation}\label{eq-b44} R_\nu:= \frac {|z_\nu^j-z_\nu^k|}{r_\nu^j} \to \infty, \text{ as } \nu\to \infty.  
\end{equation} Then by Lemma \ref{le-bilinear-sphere},
\begin{equation}\label{eq-b45}
\begin{split}
& \left\| \int e^{i(x-x_\nu^j)y+i((t-t_\nu^j)\frac {\sqrt{1-|r_\nu^j y|^2}-1}{(r_\nu^j)^2}}\phi^j (y)dy\,\int e^{i(x-x_\nu^k)y+i\bigl(t-t_\nu^k\bigr)\frac {\sqrt{1-|r_\nu^k y|^2}-1}{(r_\nu^k)^2}}\phi^k (y)dy\right\|_{L^{1+\frac 2d}}\\
&\le C R_\nu^{-\alpha} \|\phi^j\|_2\|\phi^k\|_2 \to 0. 
\end{split}
\end{equation} Then by interpolation with the $L^\infty_{t,x}$ estimate, we see that the quantity in \eqref{eq-b43} converges to zero as $\nu\to \infty$.
\end{proof}

\section{A synthesis}\label{earthworm4}
In this section, we will give a complete profile decomposition for $\{f_\nu\}_{\nu}\in L^2(\Gamma, \sigma)$, $d\ge 2$ where $\|f_\nu\|_{L^2(\Gamma, \sigma)}=1$. We will use the idea of a system of equations to express the decomposition. 
\begin{itemize}
\item[1.] As in Proposition \ref{prop-first-decomp}, 
\begin{equation}\label{eq-b46}
f_\nu = \sum_{j=1}^N f_\nu^j +e^N, 
\end{equation} where each $f_\nu^j$ is supported on a cap $\mathcal{C}(z_\nu^j, r_\nu^j)$ with the other properties listed there. 
\item[2.] Motivated by the decomposition \eqref{eq-b5}, as in Proposition \ref{prop-decomp}, 
\begin{equation}\label{eq-b47}
f_\nu^j(x)  = \left( (1-|\cdot|^2)^{1/4} (r_\nu^j)^{-\frac d2} g_\nu^j (\frac {\cdot}{r_\nu^j})\right) \circ L_{z_\nu^j}\circ \Pi_{H_{z_\nu^j}}(x),
\end{equation}
where $g_\nu^j (y) :=\dfrac {(r_\nu^j)^{d/2} f_\nu^j \circ \Pi^{-1}_{H_{z_\nu^j}} L^{-1}_{z_\nu^j}(r_\nu^j y)}{(1-|r_\nu^j y|^2)^{1/4}}$ and $L_z, \Pi_{H_z}$ are defined in Definition \ref{rescaled-maps}. 
\item[3.] As in Proposition \ref{prop-decomp}, 
\begin{equation}\label{eq-b48}
g_\nu^j(y)=\sum_{\alpha=1}^{A_j} e^{\frac {it_\nu^{j,\alpha}|y|^2}{2}} e^{-i x_\nu^{j,\alpha}\cdot y} \phi^{j, \alpha} (y)+e_\nu^{A_j}. 
\end{equation}
\item[4.] Combining all these, we have 
\begin{equation}\label{eq-b49}
\begin{split}
f_\nu(x) &= \sum_{j=1}^N \sum_{\alpha=1}^{A_j} \left[(1-|\cdot|^2)^{\frac 14}(r_\nu^j)^{-\frac d2}\left( e^{\frac {it_\nu^{j,\alpha}|y|^2}{2}} e^{-i x_\nu^{j,\alpha}\cdot y} \phi^{j, \alpha} \right)(\frac {\cdot}{r_\nu^j})\right]\circ L_{z_\nu^j}\circ \Pi_{H_{z_\nu^j}}(x)\\
&\qquad\qquad\qquad  +w_{\nu}^{N, A_1, \cdots A_N}(x),
\end{split}
\end{equation}
where \begin{equation}\label{eq-b50}
w_{\nu}^{N, A_1, \cdots A_N} (x):= e_\nu^N+ \sum_{j=1}^{N} \tilde{e}_\nu^{j, A_j}(x).
\end{equation} Here $\tilde{e}_\nu^{j, A_j}(x) = \left[(1-|\cdot|^2)^{\frac 14}(r_\nu^j)^{-\frac d2} e^{A_j}_\nu (\frac {\cdot}{r_\nu^j})\right]\circ L_{z_\nu^j}\circ \Pi_{H_{z_\nu^j}} (x). 
$
\item[5.] The family $\Gamma_\nu^{j,\alpha}:=(r_\nu^j,z_\nu^j, x_\nu^{j,\alpha}, t_\nu^{j,\alpha})_{1\le j \le N, \, 1\le \alpha\le A_j}$ and $\Gamma_\nu^{k,\beta}:=(r_\nu^k,z_\nu^k, x_\nu^{k,\beta}, t_\nu^{k,\beta})_{1\le k \le N, \, 1\le \beta\le A_k}$ is pairwise orthogonal: 
\begin{equation}\label{eq-51}
\begin{split}
&\forall j\neq k, \,\frac {r_\nu^j}{r_\nu^k}+ \frac {r_\nu^k}{r_\nu^j} + \frac {|z_\nu^j-z_\nu^k|}{r_\nu^j} \to \infty, \text{ as }\nu\to \infty, \\
&\forall \alpha\neq \beta,  |x_\nu^{j,\alpha}-x_\nu^{j,\beta}| +|t_\nu^{j,\alpha}-t_\nu^{j,\beta} |\to \infty, \text{ as } \nu \to \infty. 
\end{split}
\end{equation}
\item[6.] Regarding the error terms, we have the following two estimates, 
\begin{align}
& \|w_\nu^{N, A_1, \cdots, A_N}\|^2_{L^2(S^d, \sigma)} := \sum_{j=1}^N \|\tilde{e}_\nu^{j, A_j}\|^2_{L^2(S^d, \sigma)} + \|e_\nu^N\|^2_{L^2(S^d, \sigma)} = \sum_{j=1}^N \|e_\nu^{j, A_j}\|^2_{L^2(S^d, \sigma)} + \|e_\nu^N\|^2_{L^2(S^d, \sigma)} , \label{eq-b52} \\
& \limsup_{\nu\to \infty} \|\widehat{w_\nu^{N, A_1, \cdots, A_N}\sigma}\|_{L^{2+\frac 4d}(\mathbb{R}\times \mathbb{R}^d)} \to 0, \text{ as } \min_{1\le j\le N} \{N, A_j\}\to \infty.  \label{eq-b53}
\end{align}Equation \eqref{eq-b52} is due to the disjoint supports of the functions on the sphere. For the second, indeed, let $\delta>0$ be an arbitrary small number. Take $N_0$ such that for every $N\ge N_0$, 
\begin{equation}\label{eq-b54}
\limsup_{\nu\to \infty} \|\widehat{e_\nu^N\sigma}\|_{L^{2+\frac 4d}(\mathbb{R}\times \mathbb{R}^d) } \le \frac \delta 3. 
\end{equation}
For every $N\ge N_0$, there exists $B_N$ such that for $A\ge B_N$, 
\begin{equation}\label{eq-55}
\limsup_{\nu\to \infty} \|\widehat{\tilde{e}_\nu^{j, A}\sigma}\|_{L^{2+\frac 4d}(\mathbb{R}\times \mathbb{R}^d) } \le \frac {\delta }{3N}. 
\end{equation}
Then the reminder $w_\nu^{N, A_1, \cdots, A_N}$ can be rewritten in the form 
\begin{equation}\label{eq-b56}
w_\nu^{N, A_1, \cdots, A_N} =e_\nu^N+\sum_{1\le j \le N} w_\nu^{j, A_j \vee B_N} +S_\nu^{N, A_1, \cdots, A_N}. 
\end{equation}Here $A_j\vee B_N = \max\{A_j, B_N\}$ and 
\begin{equation}\label{eq-b57}
S_\nu^{N, A_1, \cdots, A_N} =\sum_{1\le j\le N, \atop A_j<B_N} (w_\nu^{j, A_j} -w_\nu^{j, B_N}),
\end{equation}which is a sum of profiles in the first part of the decomposition in \eqref{eq-b49}. Therefore 
\begin{equation}\label{eq-b58}
\|\widehat{w_\nu^{N, A_1,\cdots, A_N}\sigma}\|_{L^{2+\frac 4d}(\mathbb{R}\times \mathbb{R}^d) } \le \frac {2\delta}{3} +\|\widehat{S_\nu^{N, A_1, \cdots, A_N} }\|_{L^{2+\frac 4d}(\mathbb{R}\times \mathbb{R}^d)} \le \delta,
\end{equation}for sufficiently large $\nu$ due to Proposition \ref{prop-ortho} and Proposition \ref{prop-full-decomp}. 
\end{itemize}
To summarize the above, we have proved a complete profile decomposition. 
\begin{proposition}\label{prop-profiles}
Let $\{f_\nu\}_{1\le \nu\le \infty} \in L^2(\Gamma, \sigma)$ be a sequence of functions with $\|f_\nu\|_{L^2(\Gamma,\sigma)} =1$. Then there exists a sequence of  parameters $$\Gamma_\nu^{j,\alpha}:=(r_\nu^j,z_\nu^j, x_\nu^{j,\alpha}, t_\nu^{j,\alpha})_{1\le j <\infty, \, 1\le \alpha\le A_j} \in (0,\infty)\times \Gamma \times \mathbb{R}^d \times \mathbb{R}$$
satisfying \eqref{eq-51} and $\{ \phi^{j, \alpha}\}_{1\le j <\infty, \, 1\le \alpha\le A_j} \in L^2(\mathbb{R}^d) $ such that 
\begin{equation}\label{eq-59}
\begin{split}
f_\nu(x) &= \sum_{j=1}^N \sum_{\alpha=1}^{A_j} \left[(1-|\cdot|^2)^{\frac 14}(r_\nu^j)^{-\frac d2}\left( e^{\frac {it_\nu^{j,\alpha}|y|^2}{2}} e^{-i x_\nu^{j,\alpha}\cdot y} \phi^{j, \alpha} \right)(\frac {\cdot}{r_\nu^j})\right]\circ L_{z_\nu^j}\circ \Pi_{H_{z_\nu^j}}(x)\\
&\qquad\qquad\qquad  +w_{\nu}^{N, A_1, \cdots A_N}(x). 
\end{split}
\end{equation}
We have the following orthogonality identities. 
\begin{align}
&\|f_\nu\|^2_{L^2(S^d,\sigma)} =\sum_{j=1}^N \sum_{\alpha=1}^{A_j} \|\phi^{j,\alpha}\|^2_{L^2(\mathbb{R}^d)} + \|w_\nu^{N, A_1, \cdots, A_N}\|^2_{L^2(S^d, \sigma)} +\|e_\nu^N\|^2_{L^2(S^d,\sigma)}, \text{ as } \nu\to \infty.  \label{eq-60} \\
&\text{For }j\neq k, \left\| \mathcal{F}\left( \left[(1-|\cdot|^2)^{\frac 14}(r_\nu^j)^{-\frac d2}\left( e^{\frac {it_\nu^{j,\alpha}|y|^2}{2}} e^{-i x_\nu^{j,\alpha}\cdot y} \phi^{j, \alpha} \right)(\frac {\cdot}{r_\nu^j})\right]\circ L_{z_\nu^j}\circ \Pi_{H_{z_\nu^j}} \sigma\right)\right. \label{eq-b61} \\
& \times \left. \mathcal{F}\left( \left[(1-|\cdot|^2)^{\frac 14}(r_\nu^k)^{-\frac d2}\left( e^{\frac {it_\nu^{k,\beta}|y|^2}{2}} e^{-i x_\nu^{k,\beta}\cdot y} \phi^{k, \beta} \right)(\frac {\cdot}{r_\nu^k})\right]\circ L_{z_\nu^k}\circ \Pi_{H_{z_\nu^k}}\sigma\right)\right\|_{L^{1+\frac 2d}_{t,x}(\mathbb{R}\times \mathbb{R}^d)} \to 0, \text{ as } \nu\to \infty. \notag \\
& \limsup_{\nu\to \infty}  \|\mathcal{F} \left( w_\nu^{N, A_1, \cdots, A_N}\sigma\right)\|_{L^{2+\frac 4d}(\mathbb{R}\times \mathbb{R}^d)} \to 0, \text{ as } \min_{1\le j\le N}\{N, A_j\}\to \infty. \label{eq-b62}
\end{align}
Here $\mathcal{F}(f\sigma) (x):=\widehat{f\sigma}(x)= \int_{S^d} e^{ix\cdot \xi} f(\xi)d\sigma(\xi)$. 
\end{proposition}
One consequence of \eqref{eq-b61} is the following orthogonality result. The proof follows similar lines as in \cite[Lemma 5.5]{Begout-Vargas:2007:profile-schrod-higher-d}. 
\begin{proposition}\label{le-fish-cute}
Let $N$ and $\{A_j\}_{1\le j\le N}$ be given. 
\begin{align}
& \left\| \sum_{j=1}^N \sum_{1\le \alpha \le A_j} \mathcal{F}\left( \left[(1-|\cdot|^2)^{\frac 14}(r_\nu^j)^{-\frac d2}\left( e^{\frac {it_\nu^{j,\alpha}|y|^2}{2}} e^{-i x_\nu^{j,\alpha}\cdot y} \phi^{j, \alpha} \right)(\frac {\cdot}{r_\nu^j})\right]\circ L_{z_\nu^j}\circ \Pi_{H_{z_\nu^j}}\sigma \right) \right\|^{2+\frac 4d}_{L^{2+\frac 4d}_{t,x} (\mathbb{R}\times \mathbb{R}^d)} \label{eq-b63}\\
&\le  \sum_{j=1}^N \sum_{1\le \alpha \le A_j}  \left\| \mathcal{F}\left( \left[(1-|\cdot|^2)^{\frac 14}(r_\nu^j)^{-\frac d2}\left( e^{\frac {it_\nu^{j,\alpha}|y|^2}{2}} e^{-i x_\nu^{j,\alpha}\cdot y} \phi^{j, \alpha} \right)(\frac {\cdot}{r_\nu^j})\right]\circ L_{z_\nu^j}\circ \Pi_{H_{z_\nu^j}}\sigma \right) \right\|^{2+\frac 4d}_{L^{2+\frac 4d}_{t,x} (\mathbb{R}\times \mathbb{R}^d)} \notag
\end{align}
as $\nu\to\infty.$
\end{proposition}

\section{Conditional existence}\label{Copernicus-meat}
We assume that $\mathcal{R}>\mathcal{R}_\textbf{P}$. This is actually proved in \cite{Frank-Lieb-Sabin:2007:maxi-sphere-2d} by making a real hypothesis that Gaussians are maximizers for the Strichartz inequality for the Schr\"odinger equation. Let us recall one notation from \cite{Frank-Lieb-Sabin:2007:maxi-sphere-2d}. Define the optimal constant for the Strichartz inequality for the Schr\"odinger equation, we see that
\begin{equation}\label{schrodinger}
\mathcal{R}_\mathbf{P} = (2\pi)^{-\frac 12 }\dfrac {\left( \int_{\mathbb{R}^{d+1}} |e^{it\Delta /2} \phi_G(x) |^{2+\frac 4d} dx dt \right)^{\frac {d}{2(d+2)}} }{\|\phi_G\|_{L^2(\mathbb{R})}},
\end{equation} where $\phi_G (x) = e^{-|x|^2/2}$. The \cite[Lemma 7.2]{Frank-Lieb-Sabin:2007:maxi-sphere-2d} shows that $$ \mathcal{R}>\mathcal{R}_{\mathbf{P}}. $$

We first note that when $r_\nu\to 0$, for a compactly supported and smooth function $\phi^j$, 
\begin{equation}\label{eq-b450}
 \int e^{ixy-\frac {it|y|^2}{2}}e^{it\bigl(\frac {\sqrt{1-r_\nu^2 |y|^2}-1}{r_\nu^2}+\frac {|y|^2}{2}\bigr)}\phi^j(y)dy  \to e^{it\Delta/2}\phi^j, 
 \end{equation}
in the $ L^{2+\frac d4}_{t,x} $  Strichartz norm sense, and in the pointwise sense. This follows from the standard stationary phase estimates as in \eqref{eq-b17}, \eqref{eq-b18} and \eqref{eq-b19} and the dominated convergence theorem. 

Now we begin our argument for the conditional existence. By Propositions \ref{prop-profiles} and \ref{le-fish-cute},
\begin{align}
&\limsup_{\nu\to\infty} \|\mathcal{F}(f_\nu \sigma)\|_{L^{2+\frac 4d}_{t,x} (\mathbb{R}\times \mathbb{R}^d)}^{2+4/d} \label{eq-b31} \\
&\le  \sum_{j=1}^N \sum_{1\le \alpha \le A_j} \limsup_{\nu\to\infty} \left\| \mathcal{F}\left( \left[(1-|\cdot|^2)^{\frac 14}(r_\nu^j)^{-\frac d2}\left( e^{\frac {it_\nu^{j,\alpha}|y|^2}{2}} e^{-i x_\nu^{j,\alpha}\cdot y} \phi^{j, \alpha} \right)(\frac {\cdot}{r_\nu^j})\right]\circ L_{z_\nu^j}\circ \Pi_{H_{z_\nu^j}}\sigma \right) \right\|^{2+\frac 4d}_{L^{2+\frac 4d}_{t,x} (\mathbb{R}\times \mathbb{R}^d)} \notag \\
&\le \sum_{j} \limsup_{\nu\to\infty}  \left\|\int e^{i(x-x_\nu^j)y-\frac {i(t-t_\nu^j)|y|^2}{2}}e^{i(t-t_\nu^j)\bigl(\frac {\sqrt{1-(r_\nu^j)^2 |y|^2}-1}{(r_\nu^j)^2}+\frac {|y|^2}{2}\bigr)}\phi^j(y)dy \right\|_{L^{2+\frac 4d}_{t,x} (\mathbb{R}\times \mathbb{R}^d)}^{2+4/d} \notag\\
&=\sum_{j}\limsup_{\nu\to\infty}  \left\|\int e^{ixy-\frac {it|y|^2}{2}}e^{it\bigl(\frac {\sqrt{1-(r_\nu^j)^2 |y|^2}-1}{(r_\nu^j)^2}+\frac {|y|^2}{2}\bigr)}\phi^j(y)dy \right\|_{L^{2+\frac 4d}_{t,x} (\mathbb{R}\times \mathbb{R}^d)}^{2+4/d}. \notag
\end{align}
Here we are abusing the notations of summations when passing from the second line to the third line and also increase the sum to the infinite terms. We consider two cases. If $r_\nu^j\to 0$ for some $j$,
\begin{equation}\label{eq-b32}
\begin{split}
&\limsup_{\nu\to\infty}  \left\|\int e^{ixy-\frac {it|y|^2}{2}}e^{it\bigl(\frac {\sqrt{1-(r_\nu^j)^2 |y|^2}-1}{(r_\nu^j)^2}+\frac {|y|^2}{2}\bigr)}\phi^j(y)dy \right\|_{L^{2+\frac 4d}_{t,x} (\mathbb{R}\times \mathbb{R}^d)}^{2+4/d}\\
&=\|e^{\frac {it\Delta}{2}}\phi^j\|_{2+4/d}^{2+4/d}\le \mathcal{R}^{2+4/d}_{\textbf{P}}\|\phi^j\|_2^{2+4/d} \\ &\le\mathcal{R}^{2+4/d}_{\textbf{P}}\|\phi^j\|_2^{2+4/d}<\mathcal{R}^{2+\frac 4d} \|\phi^j\|^{2+\frac 4d}_{L^2},
\end{split}
\end{equation}by the assumption that $\mathcal{R}>\mathcal{R}_{\mathbf{P}}$. On the other hand, if $r_\nu^j\to r^j$ for some $j$, by the Strichartz estimate,
\begin{equation}\label{eq-b33}
\begin{split}
&\limsup_{\nu\to\infty}\left\|\int e^{ixy-\frac {it|y|^2}{2}}e^{it\bigl(\frac {\sqrt{1-(r_\nu^j)^2 |y|^2}-1}{(r_\nu^j)^2}+\frac {|y|^2}{2}\bigr)}\phi^j(y)dy \right\|_{2+4/d}^{2+4/d}\\
&=\limsup_{\nu\to\infty}\left\|\int e^{ixy+it\frac {\sqrt{1-(r_\nu^j)^2 |y|^2}}{(r_\nu^j)^2}}\phi^j(y)dy \right\|_{2+4/d}^{2+4/d}\\
&=\limsup_{\nu\to\infty}\left\|\int e^{ixy+it\sqrt{1-|y|^2}}\left((r_\nu^j)^{-d/2}\phi^j((r_\nu^j)^{-1}y)\right)dy \right\|_{2+4/d}^{2+4/d} \\
&=\limsup_{\nu\to\infty}\left\|\int e^{ixy+it\sqrt{1-|y|^2}}\left((r_\nu^j)^{-d/2}\phi^j((r_\nu^j)^{-1}y)\sqrt{1-|y|^2}\right)\frac {dy}{\sqrt{1-|y|^2}} \right\|_{2+4/d}^{2+4/d}\\
&=\limsup_{\nu\to\infty}\left\|\int e^{ixy+it\sqrt{1-|y|^2}}\left((r^j)^{-d/2}\phi^j((r^j)^{-1}y)\sqrt{1-|y|^2}\right)\frac {dy}{\sqrt{1-|y|^2}} \right\|_{2+4/d}^{2+4/d}\\
&=\mathcal{R}^{2+4/d} \|(r^j)^{-d/2}\phi^j((r^j)^{-1}y)\sqrt{1-|y|^2}\|_{L^2(\sigma)}^{2+4/d}\\
&=\mathcal{R}^{2+4/d} \|\phi^j(y)(1-|r^jy|^2)^{1/4}\|_{L^2}^{2+4/d}\le \mathcal{R}^{2+4/d} \|\phi^j\|_{L^2}^{2+4/d}.
\end{split}
\end{equation}
Here we have used the Tomas-Stein inequality for the sphere and the dominated convergence theorem in passing the fourth line to the fifth line above. So combining $\eqref{eq-b32}$ and $\eqref{eq-b33}$, we see that
\begin{equation}\label{eq-turtle}
\begin{split}
& \mathcal{R}^{2+4/d} \le   \mathcal{R}^{2+4/d} \sum_{j \text{ in } \eqref{eq-b32}} \|\phi^j\|^{2+\frac {4}{d}}_{L^2}+ \mathcal{R}^{2+4/d} \sum_{j \text{ in } \eqref{eq-b33}} \|\phi^j\|_2^{2+4/d}  \\
&\le \mathcal{R}^{2+4/d} \sum_{j} \|\phi^j\|^{2+\frac {4}{d}}_{L^2}  \le  \mathcal{R}^{2+4/d} \left(\sum_{j} \|\phi^j\|_2^2\right)^{1+2/d}\le \mathcal{R}^{2+4/d}.
\end{split}
\end{equation}
Then $\mathcal{R}^{2+4/d}=\mathcal{R}^{2+4/d}$ forces all the inequalities above to be equal. On the other hand, because $\sum_j \|\phi^j\|_{L^2}^2\le 1$, there will be only one $j$ left from the sharpness of embedding of $\ell^{1+2/d}$ into $\ell^1$. If $r_\nu^j\to r^j>0$ for this $j$, we return to \eqref{eq-b31} and \eqref{eq-b33},  and obtain that $(r^j)^{-d/2}\phi((r^j)^{-1}y)\sqrt{1-|y|^2}$ is an extremal as desired, which is smooth.  If $r_\nu^j\to 0$ for this $j$, then we return to \eqref{eq-b31} and \eqref{eq-b32}, and obtain a contradiction to the assumption $\mathcal{R}>\mathcal{R}_\textbf{P}$. Indeed, because there is one term $\phi^j$ left, 
\begin{align} 
& f_\nu=\left[(1-|\cdot|^2)^{\frac 14}(r_\nu)^{-\frac d2}\left( e^{\frac {it_\nu |y|^2}{2}} e^{-i x_\nu\cdot y} \phi^j \right)(\frac {\cdot}{r_\nu})\right]\circ L_{z_\nu}\circ \Pi_{H_{z_\nu}},\label{eq-Salmon1}\\
&\limsup_{\nu\to \infty} \|\mathcal{F}(f_\nu\sigma)\|^{2+\frac 4d}_{L^{2+\frac 4d}_{t,x}} = \mathcal{R}^{2+\frac 4d} \limsup_{\nu\to\infty} \|f_\nu\|^{2+\frac 4d}_{L^2(S^d, \sigma)} = \mathcal{R}^{2+\frac 4d} \|\phi^j\|^{2+\frac 4d}_{L^2(\mathbb{R}^d)}. \label{eq-Salmon2}
\end{align}

\bibliography{C:/Work/Research/refs}

\begin{thebibliography}{}

\end{thebibliography}


\begin{thebibliography}{20}

\bibitem{Bahouri-Gerard:1999:profile-wave}
H.~Bahouri and P.~G{\'e}rard.
\newblock High frequency approximation of solutions to critical nonlinear wave
  equations.
\newblock {\em Amer. J. Math.}, 121(1):131--175, 1999.

\bibitem{Begout-Vargas:2007:profile-schrod-higher-d}
P.~B{\'e}gout and A.~Vargas.
\newblock Mass concentration phenomena for the {$L\sp 2$}-critical nonlinear
  {S}chr\"odinger equation.
\newblock {\em Trans. Amer. Math. Soc.}, 359(11):5257--5282, 2007.

\bibitem{Bennett-Bez-Carbery-Hundertmark: heat-flow} 
J.~Bennett, N. ~Bez, A.~Carbery, and D. ~Hundertmark.
\newblock Heat flow of monotonicity of Strichartz norms. 
\newblock {\em Analysis and PDE.}, Volume 2, No. 2, 2009, 147-- 158.

\bibitem{Bourgain:1998:refined-Strichartz-NLS}
J.~Bourgain.
\newblock Refinements of {S}trichartz' inequality and applications to
  {$2$}{D}-{NLS} with critical nonlinearity.
\newblock {\em Internat. Math. Res. Notices}, (5):253--283, 1998.


\bibitem{Carles-Keraani:2007:profile-schrod-1d}
R.~Carles and S.~Keraani.
\newblock On the role of quadratic oscillations in nonlinear {S}chr\"odinger
  equations. {II}. {T}he {$L\sp 2$}-critical case.
\newblock {\em Trans. Amer. Math. Soc.}, 359(1):33--62 (electronic), 2007.


\bibitem{COeS15}
E.~Carneiro and D.~Oliveira e Silva.
\newblock Some sharp restriction inequalities on the sphere.
\newblock {\em Int. Math. Res. Not. IMRN}, (2015), no.~17, 8233--8267.

\bibitem{COeSS19}
E.~Carneiro, D.~Oliveira e Silva, and M.~Sousa.
\newblock Sharp mixed norm spherical restriction.
\newblock {\em Adv. Math.}, \textbf{341} (2019), 583--608.

\bibitem{Carneiro2019}
  E. ~Carneiro and D. ~Oliveira e Silva and M.~Sousa.
\newblock Extremizers for Fourier restriction on hyperboloids.
 \newblock{\em Annales de l'I.H.P. Analyse non linéaire}, Volume 36, No. 2, 2019, 389--415. 

\bibitem{Christ-Shao:extremal-for-sphere-restriction-I-existence}
M.~Christ and S. ~Shao.
\newblock {Existence of extremals for a Fourier restriction inequality}.
\newblock {\em Analysis and PDE,} 5(2): 261-312, 2012.


\bibitem{Christ-Shao:extremal-for-sphere-restriction-II-characterizations}
M.~Christ and S.~Shao.
\newblock { On the extremizers for an adjoint Fourier restriction inequality}.
\newblock {\em Advance in Mathematics,} Volume 230, Issue 3, 20 June 2012, Pages 957-977.

\bibitem{Foschi:2007:maxi-strichartz-2d}
D.~Foschi.
\newblock Maximizers for the {S}trichartz inequality.
\newblock {\em J. Eur. Math. Soc. (JEMS)}, 9(4):739--774, 2007.

\bibitem{Folland:1999}
G.~Folland.
\newblock Real Analysis: Modern Techniques and Their Applications, 2nd Edition.
\newblock {\em Willey}, 1999. 

\bibitem{Fo15}
D.~Foschi.
\newblock Global maximizers for the sphere adjoint Fourier restriction inequality.
\newblock {\em J. Funct. Anal.}, Volume 268, (2015), 690--702.

\bibitem{Frank-Lieb-Sabin:2007:maxi-sphere-2d}
R. ~Frank, E.~Lieb and J. ~Sabin. 
\newblock Maximizers for the {T}omas-{S}tein inequality.
\newblock {\em Geometric and functional Analysis}, Volume 26, 1095--1134, (2016).

\bibitem{Hundertmark-Zharnitsky:2006:maximizers-Strichartz-low-dimensions}
D.~Hundertmark and V.~Zharnitsky.
\newblock On sharp {S}trichartz inequalities in low dimensions.
\newblock {\em Int. Math. Res. Not.}, pages Art. ID 34080, 18, 2006.

\bibitem{Killip-Visan:2008:clay-lecture-notes}
R.~Killip and M.~Visan.
\newblock Nonlinear {S}chr\"odinger equations at critical regularity.
\newblock {\em Lecture notes for the summer school of Clay Mathematics
  Institute}, 2008.
  
 \bibitem{kunze}
M.~Kunze.
\newblock {On the existence of a maximizer for the Strichartz inequality}.
\newblock {\em Comm. Math. Phys.} 243 (2003), 137-162.

\bibitem{Lieb-Loss:2001}
E.~Lieb, M. ~Loss.
\newblock {\em Analysis, 2nd edition}.
\newblock {\em Graduate Studies in Mathematics}, v. 14, 2001.

\bibitem{lions1984a}
P.-L.~Lions.
\newblock {The concentration-compactness principle in the calculus of variations. The locally compact case. I}.
\newblock {\em Ann. Inst. H. Poincar\'e Anal. Non Lin\'eaire 1}, (1984), no. 2, 109--145.

\bibitem{lions1984b}
\bysame.
\newblock {The concentration-compactness principle in the calculus of variations. The locally compact case. II}.
\newblock {\em Ann. Inst. H. Poincar\'e Anal. Non Lin\'e aire 1},  (1984), no. 4, 223--283.

\bibitem{lions1985a}
\bysame.
\newblock {The concentration-compactness principle in the calculus of variations. The limit case. I}.
\newblock {\em Rev. Mat. Iberoamericana 1}, (1985), no. 1, 145--201.

\bibitem{lions1985b}
\bysame.
\newblock {The concentration-compactness principle in the calculus of variations. The limit case. II}.
\newblock {\em Rev. Mat. Iberoamericana 1}, (1985), no. 2, 45--121.

\bibitem{Merle-Vega:1998:profile-schrod}
F.~Merle and L.~Vega.
\newblock Compactness at blow-up time for {$L\sp 2$} solutions of the critical nonlinear {S}chr\"odinger equation in 2{D}.
\newblock {\em Internat. Math. Res. Notices}, (8):399--425, 1998.

\bibitem{Moyua-Vargas-Vega:1999}
A.~Moyua, A.~Vargas, and L.~Vega.
\newblock Restriction theorems and maximal operators related to oscillatory integrals in {$\mathbb{R}\sp 3$}.
\newblock {\em Duke Math. J.}, 96(3):547--574, 1999.

  
\bibitem{Quilodran:2015}
R.~Quilodr\'an.
\newblock{Nonexistence of extremals for the adjoint restriction inequality on the hyperboloid}.
 \newblock{\em J. Anal.Math.}, 125, 2015, 37--70.
  
\bibitem{Shao:2009profilesAiry}
S.~Shao. 
\newblock The linear profile decomposition for the Airy equation and the existence of maximizers for the Airy Strichartz inequality. 
\newblock {\em Analysis \& PDE,} Vol. 2 (2009), No. 1, 83--117.

\bibitem{Shao:2009} 
S.~Shao.
\newblock {M}aximizers for the Strichartz and the {S}obolev-{S}trichartz inequalities for the Schr\"odinger equation.
\newblock {\em Electron. J. Differential Equations (2009)}, No.~3, 13 pp.

\bibitem{Shao:2016TS}
S.~Shao. 
\newblock On existence of extremizers for the Tomas-Stein inequality for $S^1$. 
\newblock {\em Journal of Functional Analysis,} 270 (2016), 3996-4038.


\bibitem{Stein:singular integrals}
E.~Stein.
\newblock {\em Singular Integrals and differentiability properties of functions}.
\newblock Princeton University Press, Princeton, N.J., 1970.
\newblock {\em Princeton Mathematical Series, No. 31.}


\bibitem{Stein-Weiss:1971:fourier-analysis}
E.~Stein and G.~Weiss.
\newblock {\em Introduction to {F}ourier analysis on {E}uclidean spaces}.
\newblock Princeton University Press, Princeton, N.J., 1971.
\newblock {\em Princeton Mathematical Series, No. 32.}

\bibitem{Stein:1993}
E.~Stein.
\newblock {\em Harmonic analysis: real-variable methods, orthogonality, and
  oscillatory integrals}, volume~43 of {\em Princeton Mathematical Series}.
\newblock {\em Princeton University Press}, Princeton, NJ, 1993.
\newblock With the assistance of Timothy S. Murphy, Monographs in Harmonic
  Analysis, III.
  
  
\bibitem{Strichartz-1977}
R.~Strichartz.
\newblock {Restriction of Fourier transforms to quadratic surfaces and decay of solutions of wave equations}.
\newblock {\em Duke Math Journal}, Vol. 44, No. 3, 1977.

\bibitem{Tao:book}
T. Tao, 
\newblock{Nonlinear dispersive equations. Local and global analysis.} 
\newblock{\em CBMS Regional Conference Series in Mathematics}, 106.
American Mathematical Society, Providence, RI, 2006.

\bibitem{Tao:2003:paraboloid-restri}
T.~Tao.
\newblock A sharp bilinear restrictions estimate for paraboloids.
\newblock {\em Geom. Funct. Anal.}, 13(6):1359--1384, 2003.

\bibitem{Tao-Vargas-Vega:1998:bilinear-restri-kakeya}
T.~Tao, A.~Vargas, and L.~Vega.
\newblock A bilinear approach to the restriction and {K}akeya conjectures.
\newblock {\em J. Amer. Math. Soc.}, 11(4):967--1000, 1998.

\bibitem{Wolff:2001:restric-cone}
 T. ~Wolff.
  \newblock{A sharp bilinear cone restriction estimate}.
  \newblock{\em Ann. of Math. (2)} Volume 153, No. 3, 2001, 661--698. 



\bibitem{xu:2000:Funk_Hecke}
Y.~Xu.
\newblock Funk-{H}ecke formula for orthogonal polynomials on spheres and on balls.
\newblock {\em Bull. London Math. Soc.}, 32(4):447--457, 2000.










\end{thebibliography}
\bibliographystyle{plain}

\end{document}